\documentclass[pdflatex,sn-mathphys-num]{sn-jnl}

\usepackage{graphicx}%
\usepackage{multirow}%
\usepackage{amsmath,amssymb, amsthm, amsfonts}%
\usepackage{amsthm}%
\usepackage{mathrsfs}%
\usepackage[title]{appendix}%
\usepackage{xcolor}%
\usepackage{textcomp}%
\usepackage{manyfoot}%
\usepackage{booktabs}%
\usepackage{algorithm}%
\usepackage{algorithmicx}%
\usepackage{algpseudocode}%
\usepackage{listings}%
\usepackage{lipsum}
\usepackage{bbm}
\usepackage{mathtools}
\usepackage{nicematrix}
\usepackage{aligned-overset}
\usepackage{setspace}
\usepackage{enumitem}
\usepackage[T1]{fontenc}

\theoremstyle{thmstyletwo}
\newtheorem{definition}{Definition}[section]
\theoremstyle{thmstyleone}
\newtheorem{theorem}[definition]{Theorem}
\newtheorem{lemma}[definition]{Lemma} 
\newtheorem{proposition}[definition]{Proposition} 
\newtheorem{corollary}[definition]{Corollary}
\theoremstyle{thmstyletwo}
\newtheorem{example}[definition]{Example}
\newtheorem{remark}[definition]{Remark}

\newcommand{\fancyh}{\mathfrak{h}}
\newcommand{\fancyg}{\mathfrak{g}}
\newcommand{\fancym}{\mathfrak{m}}
\newcommand{\coadjoint}{\operatorname{Ad}^*}
\newcommand{\Adjoint}{\operatorname{Ad}}
\newcommand{\adjoint}{\operatorname{ad}}
\newcommand{\rdbracket}{[\mathbb{R}^d]}
\newcommand{\R}{\mathbb{R}}

\newcommand{\C}{\mathbb{C}}
\newcommand{\rank}{\operatorname{rank}}

\begin{document}
	
	\title[Article Title]{Fourier Inversion on the Group of Signatures}
	

	\author[1,2]{\fnm{Frank} \sur{Filbir}}\email{frank-dieter.filbir@helmholtz-munich.de}
    
    \author*[1]{\fnm{Davide} \sur{Nobile}}\email{davide.nobile@tum.de}
	
	\author[1]{\fnm{Marco} \sur{Rauscher}}\email{marco.rauscher@tum.de}

	\affil[1]{\orgdiv{Department of Mathematics}, \orgname{Technical University of Munich}, \orgaddress{\street{Boltzmannstr. 3}, \city{Garching bei M\"unchen}, \postcode{85748}, \country{Germany}}}
	
	\affil[2]{\orgdiv{Department of Mathematical Imaging and Data Analysis}, \orgname{Helmholtz Zentrum M\"unchen}, \orgaddress{\street{Ingolst\"adter Landstr. 1}, \city{Neuherberg}, \postcode{85764}, \country{Germany}}}

	
	\abstract{The main objective of this work is to develop a framework for Fourier analysis on the group of signatures, $G_N(\R^d)$. Employing Kirillov's orbit method, we define the Fourier transform on this group via irreducible unitary representations. Our main contribution is the derivation of necessary and sufficient conditions for identifying coadjoint orbits in general position of $G_N(\R^d)$. This enables the computation of the set of \textit{jump-indices} of generic orbits, crucial for the Fourier inversion theorem. We also obtain an explicit construction of a polarization for any linear functional in general position, which allows a concrete description of the Fourier transform of functions on $G_N(\R^d)$. This framework yields an explicit Fourier inversion formula for the group of signatures in arbitrary dimension. 
		
		Furthermore, we show that the theoretical framework developed here extends naturally to a broader class of graded Lie groups, including the group generated by the truncated tensor algebra $T_0^N(\R^d)$.}
	

	\maketitle

    \section{Introduction}\label{sec1}

In many areas of data analysis---ranging from machine learning and hypothesis testing to the approximation of solutions of stochastic differential equations---sequential data arise naturally. Such data can often be represented within spaces of continuous functions of prescribed regularity, commonly referred to as paths. A central challenge in their analysis lies in efficiently capturing the inherent temporal order of the sequence. Rather than working directly with the path itself, it is often advantageous to employ a more compact and expressive representation. The signature map offers precisely such an encoding.

The signature can be formulated for paths with different regularity. For simplicity, we restrict ourselves to continuous paths of bounded variation.  
\begin{equation*}
    BV([0,1],\R^d)\coloneqq \bigg\{x\in C([0,1],\R^d)\;\big|\;\sup_{p\in D([0,1])} \sum_{i\in p} |x_{i+1}-x_i|_{\R^d} <\infty \bigg\} ,
\end{equation*}
where $D([0,1])$ denotes the set of all finite partitions of the interval $[0,1]$. The truncated level $N$ signature maps paths $x(\cdot) \in BV([0,1],\R^d)$ to a real tensor space and is defined as follows:
\begin{equation*}
    S_N(x)_{s, t} \coloneqq\left(1, \int_{s<u<t} d x_u, \ldots, \int_{s<u_1<\ldots<u_k<t} d x_{u_1} \otimes \ldots \otimes d x_{u_k}\right)\in G_N(\R^d),
\end{equation*}
where $G_N(\R^d) \subset T^N(\R^d)\coloneqq \bigoplus_{k=1}^N (\R^d)^{\otimes k}$ forms a subset in the tensor space. We refer to Section \ref{sec: signatures} for a precise definition of this notation. Moreover, the image under the truncated signature map, denoted by $G_N(\R^d)$, forms a Lie group w.r.t. the tensor product. For the non-truncated case, signatures are denoted as $S(x)_{s,t}$ and they form an infinite-dimensional Lie group in the completed tensor algebra $T((\R^d))$ called $G(\R^d)$. 

A central result in the theory of signatures is their universality, i.e., every continuous functional on a compact subset of $BV([0,1], \R^d)$ can be uniformly approximated by linear functionals of the signature. Chevyrev and Lyons \cite{Chevyrev_2016} established that---under suitable conditions---the law of a random path is uniquely determined by its expected signature. This result was later extended by Chevyrev and Oberhauser \cite{chevyrev2022signature} to more general paths using a so-called robust signature. These results naturally induce a distance between probability measures on the path space, called the Signature Maximum Mean Discrepancy. This metric has been employed in various applications, including as a loss function in generative models \cite{ni2021sig} and as a test statistic in nonparametric hypothesis testing \cite{hao_ni2024}. 

In \cite{Chevyrev_2016} and \cite{hao_ni2024} the authors introduce and study an analogue of the classical characteristic function for random variables taking values on the group of signatures. Specifically, for a $G(\R^d)$-valued random variable $X$, the characteristic function is defined as
\begin{equation*}
    \Phi_X(M)\coloneq \mathbb E[\tilde M(X)],
\end{equation*}
where $M$ is a unitary representation of $\R^d$ and $\tilde M$ its natural extension to $G(\R^d)$. The key idea in this line of work is the use of representations satisfying a point-separating property, which ensures that the associated characteristic function uniquely determines the law of the distribution.

In practical applications, one typically works with truncated signatures, as handling the full infinite sequence is often numerically infeasible. This approach is justified by the factorial decay of iterated integrals (see, e.g., \cite{chevyrev2022signature}), which ensures that the truncated signature maintains good approximation properties.

Therefore, it is natural to seek a deeper understanding of the truncated group $G_N(\R^d)$ and the approximation behavior of signatures taking values therein. This work aims to investigate the structure of truncated signatures through the lens of harmonic analysis, establishing a foundation for decomposing functions on $G_N(\R^d)$ into components belonging to subspaces of unitary irreducible representations. In particular, the ultimate goal of this article is to specify the unitary dual $\hat G_N(\R^d)$, which contains the equivalence classes of all unitary, irreducible representations of $G_N(\R^d)$, and to develop a meaningful theory of Fourier analysis on the group of truncated signatures. 

The concept of Fourier transform can be generalized to functions on a non-commutative, locally compact group $G$ using unitary representations. In this setting, the Fourier transform of a function $f \in L^1(G)$ is defined as a map on the unitary dual $\hat{G}$ of $G$, that assigns to each equivalence class $\pi \in \hat{G}$ a linear operator defined by
\begin{equation*}
    \hat{f}(\pi) = \pi(f) = \int_G f(x)\pi(x)dx \;.
\end{equation*}
Kirillov \cite{kirillov1962unitary} showed that when $G$ is a connected, simply connected nilpotent Lie group with Lie algebra $\fancyg$, its unitary dual $\hat{G}$ can be parameterized by linear functionals in $\fancyg^*$. This means that every equivalence class of linear functionals $\ell \in \fancyg^*$ can be identified with an element $\pi \in \hat{G}$ denoted as $\pi_\ell$, where two linear functionals are equivalent whenever they lie in the same coadjoint orbit. Further, the following Fourier inversion formula holds for any Schwartz function on $G$
	 \cite[Theorem~4.3.9]{corwin1990representations}
	\begin{equation}\label{eq: intor corwin inversion formula}
		f(x) = \int_{U \cap (\fancyg^*)_T}\sqrt{\operatorname{det}D(\ell)}\operatorname{Tr}\left( \pi_\ell(x)^{-1}\pi_\ell(f)\right)d\ell \;.
	\end{equation}
	Here, $\fancyg^* \cong \R^m$ denotes the dual of the Lie algebra of $G$, $T$ is a subset of basis elements of $\fancyg^*$ and $(\fancyg^*)_T = \R$-span$\{e_i \;:\; e_i \in T\}$. The measure $d\ell$ is the standard Lebesgue measure on $(\fancyg^*)_T$, and $U$ denotes the set of coadjoint orbits in general position. This brings us to the main challenge in applying the theory, which is characterizing the orbits in general position of the group $G$. To the best of our knowledge, this has only been carried out explicitly for a few specific examples; see, for example, \cite{corwin1990representations} and \cite{kirillov2004lectures}. Moreover, explicitly constructing the representation $\pi_\ell$ typically requires finding a \textit{polarization} of $\ell \in \fancyg^*$, which is itself a difficult and often infeasible problem.
	
    The main contribution of our work is Theorem \ref{thm: characterization of gen orbits}, which establishes necessary and sufficient conditions for any coadjoint orbit of $G_N(\R^d)$ to be in general position. This characterization enables the explicit computation of the set $T$ and the matrix $D(l)$ in (\ref{eq: intor corwin inversion formula}); see Lemma \ref{lem: S and T}. Additionally, it allows us to construct a polarization explicitly for any linear functional in general position; see Corollary \ref{cor: explicit polarization}. Finally, it enables us to define the Fourier transform as a function on a subspace of $g_N^*(\R^d)$ instead of on a set of equivalence classes. Although we focus on the group of signatures, in Corollary \ref{cor: generalization of results 1} and Remark \ref{rem: generalization of results 2}, we show that all our results and proofs remain valid when the group of signatures is replaced by a more general graded Lie group satisfying certain regularity conditions. 
    
    We now outline the structure of this paper. Section \ref{sec: preliminaries} is dedicated to recalling some basic definitions and results that will be used throughout. Specifically, we recall the construction of the group of signatures and some key results from the theory of nilpotent Lie groups.
	
	In Section \ref{sec: Harmonic analysis}, we develop the framework for harmonic analysis on $G_N(\R^d)$.  We begin by constructing invariant measures on the group $G_N(\R^d)$ and its homogeneous spaces $G_N(\R^d)/H$. We then introduce the concept of induced representations, which plays a central role in Kirillov theory. Finally, we present Kirillov's orbits method in our context, which provides a complete description of the unitary dual of $G_N(\R^d)$.
	
	Section \ref{chapter: A Fourier inversion formula for the group of Signatures} is dedicated to the development of Fourier analysis for the group of signatures. We introduce the notion of orbits in general position, which is key to the formulation of the inversion theorem. The central result of this chapter is Theorem \ref{thm: characterization of gen orbits}, which gives a complete characterization of these orbits for the group $G_N(\R^d)$ for arbitrary $N$ and $d$. This characterization enables the computation of the general dimensions and jump indices of the coadjoint orbits of $G_N(\R^d)$, as well as the explicit construction of a polarization $\fancyh$ for any $l \in \fancyg_N(\R^d)^*$ in general position.

\section{Preliminaries}\label{sec: preliminaries}

\subsection{Signatures}\label{sec: signatures}
	We begin by recalling the construction of the group of signatures. For $N \in \mathbb{N}$, the $N$-step tensor algebra over $\R^d$ is defined as the real algebra
	\begin{equation}
		T^N(\mathbb{R}^d) \coloneqq \bigoplus_{k = 0}^N \;(\mathbb{R}^d)^{\otimes k} \;,
	\end{equation}
	where $(\mathbb{R}^d)^{\otimes k}$ denotes the $k$-fold tensor product of $\R^d$ and $(\mathbb{R}^d)^{\otimes 0} = \mathbb{R}$. This space becomes an algebra under component-wise addition, and a product defined by the tensor multiplication
	\begin{equation}\label{eq: deifinition of tensor product}
		p_k(g\otimes h) = \sum_{i=1}^k p_i(g) \otimes p_{k-i}(h),\ \forall k\in \{1,\dots,N\}
	\end{equation}
	for all $g,h \in T^N(\mathbb{R}^d)$, where $p_k(g)=g^k\in (\R^d)^{\otimes k}$ denotes the projection to the $k$-th component.
	We also define the subsets
	\begin{align*}
		&T_0^N\left(\mathbb{R}^d\right) = \left\{g \in T^N\left(\mathbb{R}^d\right): p_0(g)=0\right\} \cong \bigoplus_{k = 1}^N \;(\mathbb{R}^d)^{\otimes k} \quad \text{and}\\
		&T_1^N\left(\mathbb{R}^d\right) = 1+ T_0^N\left(\mathbb{R}^d\right)= \left\{g \in T^N\left(\mathbb{R}^d\right): p_0(g)=1\right\} \;.
	\end{align*}
	and, on these spaces, we define the exponential and logarithmic maps:
	\begin{align}\label{eq: exponential and logarithm}
		\exp :\; &T_0^N\left(\mathbb{R}^d\right) \rightarrow T_1^N\left(\mathbb{R}^d\right), \quad X  \mapsto 1+\sum_{k=1}^N \frac{X^{\otimes k}}{k!}\\
		\log :\; &T_1^N\left(\mathbb{R}^d\right) \rightarrow T_0^N\left(\mathbb{R}^d\right), \quad (1+X) \mapsto \sum_{k=1}^N(-1)^{k+1} \frac{X^{\otimes k}}{k} \;.
	\end{align}
	Then, $\exp$ is a bijection with inverse given by $\log$ (see \cite[Lemma~2.21]{lyons2007roughpaths}. Further, both $\exp$ and $\log$ are continuously differentiable, as they are essentially polynomial maps in the coordinates of $X$, and therefore, $\exp$ is a global diffeomorphism $T_0^N\left(\mathbb{R}^d\right) \to T_1^N\left(\mathbb{R}^d\right)$. The set  $T_1^N\left(\mathbb{R}^d\right)$ becomes a Lie group with group operation given by the tensor product defined in (\ref{eq: deifinition of tensor product}) restricted to $T_1^N\left(\mathbb{R}^d\right)$; see \cite[Proposition 7.17]{friz2010multidimensional}. Moreover, $T_0^N\left(\mathbb{R}^d\right)$ becomes a nilpotent Lie algebra, by defining the Lie bracket operation	
	\begin{equation*}
		[\cdot, \cdot]: T_0^N\left(\mathbb{R}^d\right) \times T_0^N\left(\mathbb{R}^d\right) \to T_0^N\left(\mathbb{R}^d\right), \; [X,Y] = X \otimes Y - Y \otimes X
	\end{equation*}
	as shown in \cite[Proposition 7.19]{friz2010multidimensional}. All in all, we obtain the following result.
	\begin{lemma}
		$T_1^N\left(\mathbb{R}^d\right)$ is a connected, simply connected nilpotent Lie group with Lie algebra $T_0^N\left(\mathbb{R}^d\right)$.
	\end{lemma}
	We now introduce the notion of path signatures. Recall that a path $x: [0,1] \to \R^d$ is said to be of bounded variation if
	\begin{equation*}
		\sup_{p\in D([0,1])} \sum_{i\in p} |x_{i+1}-x_i|_{\R^d} <\infty
	\end{equation*}
	where $D([0,1])$ denotes the set of all finite partitions of the interval $[0,1]$. We denote by $BV([0,1],\R^d)$ the set of all continuous paths of bounded variation $[0,1] \to \R^d$ and define the $N$-step signature of a path $x \in BV([0,1],\R^d)$ as follows.
	\begin{definition}
		Let $x: [0,1] \to \mathbb{R}^d, x(t) = (x_t^1, \ldots, x_t^d)$ be a continuous path of bounded variation. We introduce the notation
		\begin{equation*}
			\int_{0 < t_1 < \dots < t_k < 1} dx_{t_1} \otimes \dots \otimes dx_{t_k} \coloneqq \sum_{i_1, \dots, i_k} \left( \int_{0 < t_1 < \dots < t_k < 1} dx_{t_1}^{i_1} \dots dx_{t_k}^{i_k} \right) (e_{i_1} \otimes \dots \otimes e_{i_k})
			\in (\mathbb{R}^d)^{\otimes k} \;.
		\end{equation*}
		Then, the N-step signature of the path $x$ is defined as the sequence
		\begin{equation*}
			S_N(x)_{0,1} \coloneqq\left(1, \int_{1<u<1} d x_u, \ldots, \int_{0<u_1<\ldots<u_k<1} d x_{u_1} \otimes \ldots \otimes d x_{u_k}\right) \in  T_1^N(\R^d) \;.
		\end{equation*}
	\end{definition}
	We denote by $G_N(\R^d)$ the set of all signatures of continuous paths of bounded variation taking values in $\R^d$, i.e.
	\begin{equation}\label{eq: group of signatures definition}
		G_N(\mathbb{R}^d) := \left\{ S_N(x)_{0,1} : x \in BV\left( [0, 1], \mathbb{R}^d \right) \right\} \;.
	\end{equation}
	It is known that $G_N(\mathbb{R}^d)$ is a Lie subgroup of $T_1^N\left(\mathbb{R}^d\right)$; see \cite[Proposition~2.25]{lyons2007roughpaths}.
	Its Lie algebra is the free $N$-step nilpotent Lie algebra over $\R^d$, which we denote by $\fancyg_N(\R^d)$, realized as a subalgbegra of $ T_0^N\left(\mathbb{R}^d\right)$. Further, the exponential map $\exp: \fancyg_N(\R^d) \to G_N(\mathbb{R}^d)$ is a bijection, given by the the exponential map defined in (\ref{eq: exponential and logarithm}), restricted to $\fancyg_N(\R^d) \subset T_0^N\left(\mathbb{R}^d\right)$; see \cite[Theorem~7.30]{friz2010multidimensional}. Note that since $\exp$ is a homeomorphism and $\fancyg_N(\R^d)$ is a finite-dimensional vector space, it follows that $G_N(\mathbb{R}^d)$ is connected and simply connected. We summarize these observations in the following theorem.
	\begin{theorem}\label{thm: GN is group with algebra gn}
		Let $G_N(\mathbb{R}^d)$ be the group of signatures defined (\ref{eq: group of signatures definition}) and define
		\begin{equation*}
			\fancyg_N(\mathbb{R}^d) = \mathbb{R}^d \oplus [\mathbb{R}^d, \mathbb{R}^d] \oplus \cdots \oplus 
			\underbrace{[\mathbb{R}^d, [\ldots, [\mathbb{R}^d, \mathbb{R}^d]]]}_{\text{(N-1) brackets}} \subset T_0^N\left(\mathbb{R}^d\right) \;.
		\end{equation*}
		Then, $G_N(\mathbb{R}^d)$ is a connected, simply connected nilpotent Lie group with Lie algebra $\fancyg_N(\mathbb{R}^d)$ and $\exp(\fancyg_N(\mathbb{R}^d)) = G_N(\mathbb{R}^d)$.
	\end{theorem}
	
	\begin{remark}
		In what follows, we will often abbreviate $G_N(\mathbb{R}^d)$ and $\mathfrak{g}_N(\mathbb{R}^d)$ by $G_N$ and $\mathfrak{g}_N$, respectively, whenever the dimension $d$ of the underlying space $\mathbb{R}^d$ is clear from the context or irrelevant. Moreover, to reduce notational clutter, we will write the product of two elements $x, y \in G_N$ simply as $xy$ or $x \cdot y$, rather than $x \otimes y$.
	\end{remark}
	\subsection{Generalities on Nilpotent Lie Groups}
	In the previous section, we established that the group of signatures $G_N$ is a connected, simply connected nilpotent Lie group. In this section, we recall some classical results about this type of groups. For simplicity, we state the results specifically for $G_N$, although we mention that they hold more generally for any connected, simply connected nilpotent Lie group; see, for instance, \cite{corwin1990representations}. We begin with the well-known Baker-Campbell-Hausdorff (BCH) formula.
	\begin{theorem}{(Baker-Campbell-Hausdorff formula)}\label{thm: cbh}
		Let $\exp$ be the exponential map defined in (\ref{eq: exponential and logarithm}). Then, for any $X,Y \in \mathfrak g_N(\R^d)$ we have that
		\begin{equation*}
			\log \left(\exp X \cdot \exp Y \right) = X+Y+\frac{1}{2}[X,Y]+ (\textit{linear combination of commutators in } X \text{ and } Y)
		\end{equation*}
	\end{theorem}
	\begin{proof}
		See \cite[Theorem~1.2.1]{corwin1990representations}.
	\end{proof}
	The bijectivity of the exponential map allows us to identify $G_N$ with $\fancyg_N$ and, in particular, to transfer a coordinate system on the algebra to one on the group. Specifically, let $B = \{X_1, \ldots, X_n\}$ be a basis of $\fancyg_N$ and $x \in G_N$. Then, there is a unique $X \in \fancyg$ such that $x = \exp(X)$, and since $B$ is a basis there are $\alpha_1, \dots, \alpha_n \in \mathbb{R}$ such that $X = \sum_{i=1}^n\alpha_iX_i$. Hence,
	\begin{equation*}
		x = \exp(X) = \exp\left(\sum_{i=1}^n\alpha_iX_i\right)
	\end{equation*}
	so that we can identify $x$ with the vector $(\alpha_1, \ldots, \alpha_n) \in \mathbb{R}^n$. This coordinate system is often called \textit{canonical coordinates of the first kind} and can be used to construct a Haar measure on $G_N$ (see section \ref{sec: measures on nilpotent lie groups}). However, it is not the only way to transfer coordinates from $\fancyg_N$ to $G_N$. First, we need a definition.
	\begin{definition}{(Strong Malcev basis)}\label{def: strong malcev basis}
		Let $\fancyg$ be a nilpotent Lie algebra and $B = \{X_1, \ldots, X_m\}$ a basis for $\fancyg$. We say that $B$ is a \textit{strong Malcev} basis for $\fancyg$, if for all $n \in \{1, \ldots, m\}$:
		\begin{itemize}
			\item [i)] $\mathbb{R}$-$\operatorname{span}\{X_1, \ldots, X_n\}$ is a subalgebra of $\fancyg$.
			\item[ii)] $\mathbb{R}$-$\operatorname{span}\{X_1, \ldots, X_n\}$ is an ideal of $\fancyg$.
		\end{itemize}
		A basis that only satisfies condition $i)$ is called a weak Malcev basis for $\fancyg$. Note that by subalgebra we always mean a Lie subalgebra, i.e. a subalgebra closed under Lie bracket operations.
	\end{definition}
	\begin{example}
		Let 
		\[B = \{X_1^N, \ldots, X_{m_N}^N, X_1^{N-1}, \ldots, X_{m_{N-1}}^{N-1}, \ldots, X_1^1, \ldots, X_{m_1}^1\} \;,\]
		be a basis of $\fancyg_{N}$ such that the elements $X_1^k, \ldots, X_{m_k}^k$ span the $k$-th layer of $G_N$. Then, $B$ is a strong Malcev basis for $\fancyg_{N}$.
	\end{example}
	
	One of the most useful properties of strong Malcev bases is that they provide an alternative diffeomorphism between $\fancyg_N$ and $G_N$.
	
	\begin{proposition}\label{prop: bijectivity of gamma}
		Let $\{X_1, \ldots, X_m\}$ be a strong Malcev basis of $\fancyg_N$. Then, the map 
		\begin{equation*}
			\gamma: \fancyg \to G,\quad \gamma\left(\sum_{i=1}^m\alpha_iX_i\right) = \exp(\alpha_mX_m)\cdots\exp(\alpha_1X_1)
		\end{equation*}
		is a diffeomorphism. Moreover, if $\fancyh = \mathbb{R}$-$\operatorname{span}\{X_1, \ldots, X_k\}$ and $H = \exp(\fancyh)$, then 
		\begin{equation*}
			H = \gamma(\fancyh) = \exp(\mathbb{R}X_k)\cdots\exp(\mathbb{R}X_1).
		\end{equation*}
	\end{proposition}
	\begin{proof}
		See \cite[Proposition~1.2.7]{corwin1990representations}.
	\end{proof}
	
	We conclude this section by defining the so-called coadjoint action of $G_N$ on the dual of its Lie algebra $\fancyg_N^*$, which will be crucial in Kirillov's orbit method discussed in Section \ref{sec: kirillov theory}.
	\begin{definition}\label{def: adj and coadj action}
    Let $\exp$ be the exponential function on $\fancyg_N$, so that, by Theorem \ref{thm: GN is group with algebra gn}$, G_N = \exp(\fancyg_N)$. Then, we define the adjoint action of $G_N$ on $\fancyg_N$ by
		\begin{align*}
			\Adjoint : G_N\times\fancyg_N \to \fancyg_N, \quad
			(\exp X, Y) \mapsto (\Adjoint \operatorname{exp}X)Y = \sum_{k=0}^{\infty}\frac{1}{k!}(\adjoint X)^k Y \;,
		\end{align*}
		where $(\adjoint X) Y = [X,Y]$.\\
		Let $\fancyg_N^*$ be the dual of $\fancyg_N$, i.e. the set of continuous linear functionals $\mathfrak{g}_N \to \mathbb{R}$. Then, the coadjoint action of $G_N$ on $\fancyg_N^*$ is defined by $\coadjoint: G_N\times\fancyg_N^* \to \fancyg_N^*,$ $(\exp X, \ell) \mapsto (\coadjoint \exp X)\ell$, where $(\coadjoint \exp X)\ell$ is defined as
		\begin{align*}
			((\coadjoint \operatorname{exp}X)\ell)(Y) = \ell(\Adjoint \operatorname{exp}-X)Y) = \ell\left(\sum_{k=0}^{\infty}\frac{1}{k!}(\adjoint -X)^k Y\right) \quad \text{for all } Y \in \fancyg
		\end{align*}
	\end{definition}
	\section{Harmonic Analysis on the Group of Signatures}\label{sec: Harmonic analysis}
	\subsection{Measures on nilpotent Lie groups}\label{sec: measures on nilpotent lie groups}
	In this section we present some measure-theoretic facts on the group of signatures. First, we describe two known constructions of the Haar measure. Further, we show how to obtain a $G_N$-invariant measure on the homogeneous space $G_N/H$ when $H$ is a closed subgroup of $G_N$. The existence of such a measure is essential for the construction of induced representations presented in the next section. We note that, although we state the results in this section for the group of signatures, they also hold if one substitutes $G_N$ with any connected, simply connected nilpotent Lie group.
	\begin{theorem}\label{thm: construction of haar measure}
		Let $dX$ be the standard Lebesgue measure on $\fancyg_N$. Then, the push-forward measure $dx = \exp_{*}dX$ of $dX$ by $\exp$ is a left and right Haar measure on $G_N=\exp(\mathfrak g_N)$.
	\end{theorem}
	\begin{proof}
		See \cite[Proposition~16.40]{friz2010multidimensional} or \cite[Theorem~1.2.10]{corwin1990representations}.
	\end{proof}
	Using the push forward of the exponential map on $\fancyg_N$ is not the only way to obtain a Haar measure on $G_N$.  Using the map $\gamma$ defined in Proposition \ref{prop: bijectivity of gamma} provides an alternative construction of a left and right Haar measure on $G_N$. Interestingly, the Haar measure obtained in this way coincides with the one from the previous theorem.
	\begin{theorem}\label{thm: alternative haar measure}
        Let $\lambda$ be the standard Lebesgue measure on $\fancyg_N$ and $\gamma: \fancyg_N \to G_N$ the map defined in Proposition \ref{prop: bijectivity of gamma}, i.e.
        $$\gamma: \fancyg_N \to G_N,\quad \gamma\left(\sum_{i=1}^m\alpha_iX_i\right) = \exp(\alpha_mX_m)\cdots\exp(\alpha_1X_1)$$
        Then, $\gamma_{*}(\lambda) =  \exp_{*}(\lambda)$. In other words, $\gamma_{*}(\lambda)$ is a left and right Haar measure on $G_N$ that coincides with the one constructed in Theorem \ref{thm: construction of haar measure}.
    \end{theorem}
	\begin{proof}
		See \cite[Theorem~1.2.10 and Corollary 1.2.11]{corwin1990representations}.
	\end{proof}
	A concept that will be of crucial importance for the inducing construction, is that of a $G_N$-invariant measure on a homogeneous space of $G_N$.
	\begin{definition}
		Let $G$ be a locally compact group, $H \subset G$ a closed subgroup and consider the quotient space $G/H$. Then, we say that Radon measure $\mu$ on $G/H$ is $G$-invariant, if for any measurable set $E\subset G/H$ and $x \in G$ we have that $\mu(xE) = \mu(E)$.
	\end{definition}
	In particular, if $\mu$ is a $G$-invariant measure on $G/H$ and $f \in L^1(G/H)$, then 
	\begin{equation*}
		\int_{G/H} f(yx) d\mu(x) = \int_{G/H} f(x) d\mu(x) \;,
	\end{equation*}
	for any $y \in G$. The following result provides a construction of a $G_N$ invariant measure on any quotient space of $G_N$.
	\begin{theorem}\label{thm: construction of G-inv measure}
		Let $\{X_1, \ldots, X_n\}$ be a strong Malcev basis for $\fancyg_N$, $\fancyh = \R$-$\operatorname{span}\{X_1, \ldots, X_k\}$ and $H = \exp(\fancyh)$. We identify $\fancyg_N/\fancyh$ with $\R^{n-k}$ and define $\phi: \fancyg_N/\fancyh \to G_N/H$ as 
		$$\phi(\alpha_n, \ldots, \alpha_{k+1}) = \exp(\alpha_nX_n)\cdots\exp(\alpha_{k+1}X_{k+1}) \cdot H = \gamma(\alpha_n, \ldots, \alpha_{k+1}, 0, \ldots, 0) \cdot H\;,$$
		where $\gamma$ is the function from Proposition \ref{prop: bijectivity of gamma}. Then, $\phi$ is a diffeomorphism and if $\lambda$ is the Lebesgue measure on $\fancyg_N/\fancyh = \R^{n-k}$ then $\phi_*(\lambda)$ is a $G$-invariant measure on $G_N/H$.
	\end{theorem}
	\begin{proof}
		See \cite[Theorem~1.2.12]{corwin1990representations}.
	\end{proof}
	\subsection{The inducing construction}\label{sec: inducing construction}
	Recall that a unitary representation $\pi$ of a topological group $G$ is a group homomorphism from $G$ into the group of unitary operators on some Hilbert space  $\mathcal{H}_{\pi}$, denoted by $U(\mathcal{H}_{\pi})$, which is continuous with respect to the strong operator topology. In other words, $\pi$ is a map $G \to U(\mathcal{H}_{\pi})$ such that for all $x,y \in G$, $\pi(xy) = \pi(x)\pi(y)$, $\pi(x^{-1}) = \pi(x)^{-1}$ and $x \mapsto \pi(x)v$ is continuous for all $v \in \mathcal{H}_{\pi}$. 

     Two unitary representations $\pi_1: G \to U(\mathcal{H}_{\pi_1}), \pi_2: G \to U(\mathcal{H}_{\pi_2})$ of $G$ are said to be \textit{equivalent}, written $\pi_1 \cong \pi_2$, if there is a unitary operator $U: \mathcal{H}_{\pi_1} \to \mathcal{H}_{\pi_2}$, such that for all $x \in G$
    \[\pi_2(x) = U\pi_1(x)U^{-1} \;.\]

    Further, a unitary representation $\pi: G \to U(\mathcal{H}_{\pi})$ of $G$ is called \textit{irreducible} if there is no $\pi$-invariant subspace of $\mathcal{H}_{\pi}$. That is, there is no subspace $V \subsetneq \mathcal{H}_{\pi}$ with $\pi(V) \subseteq V$.
    
	We now introduce the concept of induced representations, i.e. a representation of a group $G$ constructed from a representation of a subgroup $H \subseteq G$. For simplicitly, we only present this construction in the situation relevant for our purpose. We refer to \cite{folland2015course} for a more general and detailed introduction to induced representation.
	
	Let $\ell\in \fancyg_N^*$ and $\mathfrak{h}_\ell \subset \mathfrak{g}_N$ a subalgebra such that $\ell([X,Y]) = 0$ for all $X,Y \in \mathfrak{h}_{\ell}$. Then, by the Baker-Campbell-Hausdorff formula (Theorem \ref{thm: cbh}),
	\begin{equation*}
		\ell\left(\log(\operatorname{exp}X \cdot \operatorname{exp} Y)\right)  =  \ell(X) + \ell(Y)
	\end{equation*}
	for all $X,Y \in \mathfrak{h}_\ell$. Hence, for $H_\ell = \operatorname{exp}(\mathfrak{h}_\ell) \subset G_N$, 
	\[\chi_{\ell}: H_\ell \to \mathbb{C}, \; \chi(\operatorname{exp}X) = e^{i\ell(X)}\]
	is a one-dimensional representation of the subgroup $H_\ell$, since
	\begin{equation*}
		\sigma(\exp X \cdot \exp Y) = e^{i\ell\left(\log(\operatorname{exp}X\cdot \; \operatorname{exp} Y)\right)} = e^{i\ell(\operatorname{exp}X) + i\ell(\operatorname{exp}Y)} = \sigma(\exp X) \sigma(\exp Y) \;.
	\end{equation*}
	and the map is clearly continuous. Note that we interpret the complex number $e^{i\ell(X)}$ as a unitary operator $\C \to \C$, $z \mapsto e^{i\ell(X)}z$.
	We now define $\mathcal{F}_{\ell}$ to be the space containing all functions $f: G_N \to \mathbb{C}$ such that 
	\begin{align*}
		&i)\; f(\exp X \cdot \exp Y) = e^{i\ell(Y)}f(\exp X) \text{ for all } X \in \fancyg_N, Y \in \fancyh_{\ell} \\
		i&i) \;\int_{G_N/H_{\ell}} |f(x)|^2 d\mu(xH) < \infty
	\end{align*}
	where $\mu = \phi_*(\lambda)$ is the $G_N$-invariant measure from Theorem \ref{thm: construction of G-inv measure}. Note that, by definition, for any $f, g \in \mathcal{F}_{\ell}$, 
	$f(x)\overline{g(x)}$ depends only on the cosets in $G_N/H_{\ell}$, since
	\begin{equation*}
		f(xh)\overline{g(xh)} =  e^{i\ell(\log h)}f(x)e^{-i\ell(\log h)}\overline{g(x)} = f(x)\overline{g(x)}
	\end{equation*}
	for all $x \in G_N$ and $h \in H_{\ell}$. Hence, the integral in $ii)$ is indeed well defined and, more generally, the map 
	\begin{equation*}
		(f,g): G_N/H_{\ell} \to \mathbb{R}, \quad (f,g)(xH_{\ell}) =f(x)\overline{g(x)}
	\end{equation*}
	is a well defined function on $G_N/H_{\ell}$. With this observation the space $\mathcal{F}_{\ell}$ becomes a complete Hilbert space, by defining the following inner product on it \cite[Chapter~6]{folland2015course}
	\begin{equation*}
		\langle f, g\rangle=\int_{G_N/H_{\ell}} (f,g)(xH_{\ell})d \mu(x H_{\ell}) = \int_{G_N/H_{\ell}} f(x)\overline{g(x)}d \mu(xH_{\ell})\;.
	\end{equation*}
	Finally, the induced representation $\operatorname{ind}_{H_{\ell}}^{G_N}(\sigma)$ of $G_N$, induced by the character $\chi_l$ is the unitary representation defined by
	\begin{equation*}
		\operatorname{ind}_{H_\ell}^{G}(\chi_l): G_N \to U(\mathcal{F}_\ell), x \mapsto L_x
	\end{equation*}
	where $L_x$ is the left shift operator defined by $(L_xf)(y) = f(x^{-1}y)$. Note that this representation is indeed unitary, since by $G_N$ invariance of $\mu$ we have that 
	\begin{align*}
		\left\langle L_yf, L_yg\right\rangle = \int_{G_N/H_{\ell}} f(y^{-1}x)\overline{g(y^{-1}x)}d \mu(xH_{\ell}) =  \int_{G_N/H_{\ell}} f(x)\overline{g(x)}d \mu(xH_{\ell}) =\left\langle f, g\right\rangle
	\end{align*}
	for all $y \in G_N$ and $f,g \in \mathcal{F}_\ell$.
	The functions in $\mathcal{F}_\ell$ have the following interesting geometric interpretation. Using the function $\gamma$ from Proposition \ref{prop: bijectivity of gamma} as a coordinate map, any function in $\mathcal{F}_\ell$ satisfies
	\begin{equation*}
		f(\gamma(x,h)) = f(\gamma(x,0)\gamma(0,h)) = e^{i\ell(h)}f(\gamma(x,0))
	\end{equation*}
	for all $(x,h) \in \fancyg/\fancyh_\ell \times \fancyh_\ell$. Hence, we see that geometrically, $f$ is a complex-valued function, such that its norm only depends on the equivalence classes of $G_N/H_\ell$ or, using $\gamma$, on the elements of $\fancyg/\fancyh_\ell$, while the elements in $H_\ell = \gamma(\fancyh_\ell) = \exp(\fancyh_\ell)$ only determine the rotation of the function value. 
    
    In particular, any function $f \in \mathcal{F}_\ell$ defines a function $\tilde{f} \in L^2(G/H_{\ell})$ via $\tilde{f}(\gamma(x,0)H_\ell) \coloneqq f(\gamma(x,0))$ and, conversely, every function $\tilde{f} \in L^2(G/H_{\ell})$ can be extended to a function $f \in \mathcal{F}_\ell$ by defining $f(\gamma(x,h)) \coloneqq e^{i\ell(h)}f(\gamma(x,0))$.
	
	The reason for our interest in this type of induced representation is that, by Kirillov's orbits method, these representations exhaust, up to equivalence, all irreducible unitary representations of $G_N$. Hence, they provide a method to parametrize the unitary dual $\hat{G}_N$ using linear functionals in $\fancyg_N^*$.
	
	\subsection{Kirillov theory}\label{sec: kirillov theory}
	We present Kirillov theory in its full generality, as restricting to the special case of the signatures group does not simplify the statements in any way. Kirillov originally established the results presented in this section in \cite{kirillov1962unitary} and a more self-contained exposition also appears in his book \cite{kirillov2004lectures}. Our exposition mostly follows the one provided by Corwin and Greenleaf in \cite{corwin1990representations}. The proofs of the results presented here can be found in any of the mentioned sources.
	
	In the following, $G$ is always assumed to be a connected, simply connected, nilpotent Lie group with Lie algebra $\fancyg$. We begin with a definition.
	\begin{definition}
		Let $\ell \in \mathfrak{g}^*$. Then, we say that a sub-algebra $\mathfrak{h} \subseteq{\mathfrak{g}}$ is subordinate to $\ell$ if $\ell([X,Y]) = 0$ for all $X,Y \in \mathfrak{h}$.\\
		We say that $\mathfrak{h}$ is a maximal subordinate subalgebra for $\ell$, if $\mathfrak{h}$ is subordinate to $\ell$ and for any other subalgebra $\mathfrak{h}'$ subordinate to $\ell$, $\operatorname{dim}\fancyh ' \leq \operatorname{dim}\fancyh$. A maximal subordinate subalgebra is also called a \textit{polarization} for $\ell$. 
	\end{definition}
	Given a functional $\ell\in \fancyg^*$ and a (maximal) subordinate subalgebra $\fancym$ of $\fancyg$, we set $M = \operatorname{exp}\fancym$ and 
	\[\chi_{\ell,M}: M \to \mathbb{C}, \; \chi_{\ell,M}	(\operatorname{exp}X) = e^{i\ell(X)} \;.\]
	Then, we have seen in the previous section that $\chi_{\ell,M}$ defines a unitary representation of the subgroup $M \subseteq G$.
	The next theorem guarantees the existence of a maximal subordinate subalgebra for any $\ell \in \fancyg^*$. Note that while all polarizations of a functional $\ell$ have the same dimension, there might be in general more than one choice of polarization for $\ell$. However, the next theorem also shows that, up to equivalence, the choice of maximal subordinate subalgebra, makes no difference to the induced representation.
	\begin{theorem}\label{thm: existence of mss}
		Let $\ell\in \mathfrak{g}^*$ be a linear functional. Then, there is a maximal subordinate subalgebra $\fancym$ of $\ell$, and for $M = \exp(\fancym)$, $\operatorname{ind}_M^G(\chi_{\ell,M})$ is irreducible. Further, if $\fancym, \fancym'$ are two maximal subordinate subalgebras of $\ell$, then $\operatorname{ind}_M^G(\chi_{\ell,M}) \cong \operatorname{ind}_{M'}^G(\chi_{\ell,M'})$.
	\end{theorem}
	\begin{remark}
		Since we are only interested in the equivalence classes of unitary representations, by Theorem \ref{thm: existence of mss} we can simplify our notation and write $\pi_\ell$ instead of $\operatorname{ind}_M^G(\chi_{\ell,M})$ for the unitary irreducible representation induced by the functional $\ell \in \fancyg^*$.
	\end{remark}
	The following theorem shows that every irreducible unitary representation of a simply connected nilpotent Lie group is induced from a one-dimensional representation of some subgroup. Moreover, it characterizes the equivalence classes of irreducible unitary representations using the coadjoint action defined in Definition \ref{def: adj and coadj action}.
	\begin{theorem}\label{thm: eq of unireps with induced reps}
		Let $\pi$ be an irreducible unitary representation of $G$. Then, there is an $\ell \in g^*$ such that $\pi_\ell \cong \pi$. Further, if $\ell, \ell' \in \fancyg^*$, then $\pi_\ell \cong \pi_{\ell'}$ if and only if $\coadjoint (G)\ell=\coadjoint (G)\ell'$, i.e. if and only if $\ell$ and $\ell'$ belong to the same coadjoint orbit in $g^*$.
	\end{theorem}
	\begin{remark}\label{rem: characterization of dual}
		Theorem \ref{thm: eq of unireps with induced reps} can be summarized as follows.
		Let $\hat{G}$ be the unitary dual of $G$, i.e. the set of equivalence classes of all irreducible unitary representations of $G$. Then, $\hat{G}$ is in one-to-one correspondence with the set of coadjoint orbits $O = \{\mathcal{O}_\ell \;|\; \ell \in \fancyg^*\}$, where $\mathcal{O}_\ell = \coadjoint (G)\ell$. Hence, if we define the equivalence relation $\ell \cong \ell' :\iff \mathcal{O}_\ell = \mathcal{O}_{\ell'}$, then by Kirillov's theory we know that
		\[\hat{G} = \{[\ell] \;|\; \ell \in \fancyg^*\} \;.\]
	\end{remark}
	
	Since the computation of coadjoint orbits will play a crucial role in establishing our results, the next two examples are intended to illustrate the concept and serve as simplified versions of the more general computations presented in the next sections.
	\begin{example}\label{ex: explicit orbits of g2}
		Let $\{X_1, \ldots, X_d\}$ be an orthonormal basis of $\mathbb{R}^d$ and  set $X_{ij} = [X_i,X_j]$. Then, $\{X_{ij}\}_{1\leq i<j\leq d}$ is a basis of $[\mathbb{R}^d, \mathbb{R}^d]$ so that $\{X_1, \ldots, X_d\} \cup \{X_{ij}\}_{1\leq i<j\leq d}$ is a basis of $\fancyg_2$. Let  $\{\ell_1, \ldots, \ell_{d}\}$ and $\{\ell_{ij}\}_{1\leq i<j\leq d}$ be the corresponding dual bases.\\
		Let $\ell = \sum_i \alpha_i \ell_i + \sum_{i,j} \beta_{i,j} \ell_{i,j} \in \fancyg_2^*$ and take an arbitrary element $X = \sum_{i} g_{i}X_i +  \sum_{i,j} g_{i,j} X_{i,j} \in \fancyg_2$, where $\alpha_i, \beta_{i,j}, g_{i},g_{i,j} \in \mathbb{R}$ for all $i,j$. Then, by definition of the adjoint action, we have that
		\begin{align*}
			&\Adjoint(\operatorname{exp}-X)X_i =X_i+\sum_{j=1}^d g_j[X_i, X_j] = X_i + \sum_{j=1}^d g_j X_{i,j}\\ 
			& \Adjoint(\operatorname{exp}-X)X_{i,j} = X_{i,j} \;.
		\end{align*}
		To compute the orbit of $\ell$ we need to check how the functional  $(\coadjoint \exp X)\ell$ operates on the basis elements of $\fancyg_2$. By the definition of coadjoint action and the previous equations we have that
		\begin{align*}
			&((\coadjoint \exp X)\ell)(X_i) = \ell(\Adjoint(\exp -X)X_i) = \ell\left(X_i + \sum_{j=1}^d g_j X_{i,j}\right) = \alpha_i + \sum_{j=1}^d g_j\beta_{i,j} \\
			&((\coadjoint \exp X)\ell)(X_{i,j}) = \ell(\Adjoint(\exp -X)X_{i,j}) = \ell(X_{i,j}) = \beta_{i,j}
		\end{align*}
		Therefore, writing $(\coadjoint \exp X)\ell$ using the dual basis of $\fancyg_2^*$ we obtain that the orbit of $\ell$ is given by
		\begin{align*}
			\mathcal{O}_{\ell} = (\coadjoint G) \ell &= \left\{\sum_i(\alpha_i+\sum_jg_{j}\beta_{i,j})\ell_i+\sum_{i,j}\beta_{i,j}\ell_{i,j} \;:\; (g_{1}, \ldots, g_d) \in \mathbb{R}^d\right\}
		\end{align*}
		From this formula we see that a functional $\tilde{\ell} = \sum_i \tilde{\alpha}_i {\ell}_i + \sum_{i,j} \tilde{\beta}_{i,j} {\ell}_{i,j} \in {\fancyg}_2^*$ is in the same orbit as $\ell$ if and only if, $\tilde{\beta}_{i,j}= \beta_{i,j}$ for all $i,j = 1, \ldots, d$ {and} there are $g_1, \ldots, g_d \in \R$ such that
		\begin{align}\label{eq: condition for orbit in g2}
			\tilde{\alpha}_i = \alpha_i + \sum_{j = 1}^dg_j \beta_{i,j} \quad \forall i= 1, \ldots, d 
		\end{align}
		Note that if we define the matrix $B \in \mathbb{R}^{d\times d}$ by $B_{i,j} = \beta_{i,j}$, and $\alpha = (\alpha_1, \ldots, \alpha_d)$, $\tilde{\alpha} = (\tilde{\alpha}_1, \ldots, \tilde{\alpha}_d)$, then (\ref{eq: condition for orbit in g2}) can be rewritten as
		\begin{align}
			B\cdot g = \tilde{\alpha}-\alpha \quad \text{for some } g\in \mathbb{R}^d \;. \label{eq: conditions N2}
		\end{align}
		In particular, if the matrix $B$ is invertible, condition ($\ref{eq: conditions N2}$) is satisfied for any $\alpha, \tilde{\alpha} \in \mathbb{R}^d$. Hence, in this case, the orbit of $\ell$ is a $d$-dimensional submanifold of $\fancyg_2^*$ and it is simply given by 
		\begin{equation*}
			(\coadjoint G) \ell = \sum_{i,j} \beta_{i,j} \ell_{i,j} + \left\{\sum_{i=1}^d \alpha_i \ell_i \; | \; \alpha_1, \ldots, \alpha_d \in \mathbb{R}\right\} = \sum_{i,j} \beta_{i,j} \ell_{i,j} + \R\text{-span }\{\ell_1, \ldots, \ell_d\} \;.
		\end{equation*}
	\end{example}
	\begin{example}\label{ex: explicit orbit of g3}
		We now consider the group $G_3 = G_3(\R^d)$.\\
		Let $\{X_1^1, \ldots, X_d^1\}$, $\{X_1^2, \ldots, X_{m_2}^2\}$, $\{X_1^3, \ldots, X_{m_3}^3\}$ be bases of $\R^d$, $\rdbracket^2$ and $\rdbracket^3$ respectively. 
		Let $\ell = \sum_{i,k} \alpha_i^k\ell_i^k \in \fancyg_3^*$ and take an arbitrary $X = \sum_{i,k} g_i^kX_i^k$, then
		\begin{align*}
			\Adjoint(\exp -X)X_i^1 &= X_i^1 + [X_i^1, X] +\frac{1}{2}[X,[X_i^1,X]]\\
			&= X_i^1 + \sum_{j=1}^{d} g_j^1[X_i^1,X_j^1] + \sum_{j=1}^{m_2} g_j^2[X_i^1,X_j^2] + \frac{1}{2} \sum_{j,s = 1}^d g_j^1g_s^1[X_j^1[X_i^1,X_s^1]]\\
			\Adjoint(\exp -X)X_i^2 &= X_i^2 + [X_i^2, X] = X_i^2 + \sum_{j=1}^{d} g_j^1[X_i^2,X_j^1] \\
			\Adjoint(\exp -X)X_i^3 &= X_i^3 \;.
		\end{align*}
		Hence, using the dual basis as in the previous example, we obtain that the orbit of $\ell$ is given by
		\begin{align*}
			\mathcal{O}_{\ell}= &\Bigg\{\sum_{i=1}^d \left(\ell(X_i^1) + \ell([X_i^1, X]) +\frac{1}{2}\ell([X,[X_i^1,X]])\right)\ell_i^1\\
			&+\sum_{i=1}^{m_2} \left(\ell(X_i^2) + \ell([X_i^2, X])\right)\ell_i^2 + \sum_{i=1}^{m_3} \alpha_i^3\ell_i^3\; \Big| \; X \in \fancyg_3\Bigg\}\\
			= &\Bigg\{\sum_{i=1}^d \left(\alpha_i^1 + \sum_{j=1}^{d} g_j^1\ell([X_i^1,X_j^1]) + \sum_{j=1}^{m_2} g_j^2\ell([X_i^1,X_j^2]) + \frac{1}{2} \sum_{j,s = 1}^d g_j^1g_s^1\ell([X_j^1[X_i^1,X_s^1]]) \right)\ell_i^1\\
			&+ \sum_{i=1}^{m_2}\left(\alpha_i^2 + \sum_{j=1}^{d} g_j^1\ell([X_i^2,X_j^1])\right)\ell_i^2 + \sum_{i=1}^{m_3} \alpha_i^3\ell_i^3 \;\Big| \; (g_1^1, g_2^1, \ldots, g_{m_2}^2) \in \R^{d+m_2}\Bigg\} \;.
		\end{align*}
		Again, we see that a necessary condition for some $\tilde{\ell} = \sum_{i,k} \tilde{\alpha}_i^k \ell_i^k \in \fancyg_3^*$ to be in the same orbit as $\ell$ is that they coincide on $[\R^3]^3$. Further, there must be some $g = (g_1^1, \ldots, g_d^1) \in \mathbb{R}^d$ such that for all $i = 1, \ldots, m_2$
		\begin{equation}\label{eq: necessary condition orbit in g_3}
			\tilde{\alpha}_i^2 = \alpha_i^2 +  \sum_{j=1}^{d} g_j^1\ell([X_i^2,X_j^1])
		\end{equation}
		As before, setting $\alpha = (\alpha_1, \ldots, \alpha_{m_2})$, $\tilde{\alpha} = (\tilde{\alpha_1}, \ldots, \tilde{\alpha}_{m_2})$ and $B \in \R^{m_2 \times d}$ with $B_{i,j} = \ell([X_i^2, X_j^1])$ we can rewrite (\ref{eq: necessary condition orbit in g_3}) as
		\[B \cdot g = \tilde{\alpha} -\alpha \quad \text{for some } g \in \R^d \;.\]
		Note that this is again only a necessary, but not a sufficient condition for $\tilde{\ell}$ to be in $\mathcal{O}_{\ell}$.
	\end{example}
	
	\section{A Fourier inversion formula for the group of Signatures}\label{chapter: A Fourier inversion formula for the group of Signatures}
    In this section, we present the main contribution of our work. The first part is dedicated to orbits in general position and, specifically, to the proof of Theorem \ref{thm: characterization of gen orbits}. In the second part, we use this result to obtain an explicit description of both the Fourier transform and its inversion on $G_N$.
	\subsection{Orbits in general position}\label{sec: orbits in general position}
	Orbits in general position play a crucial role in the Fourier inversion formula, as they are the only ones that need to be considered to reconstruct a function from its transform. The main result of this section, Theorem $\ref{thm: characterization of gen orbits}$, provides necessary and sufficient conditions to characterize the general orbits of $G_N$. As a consequence, we also obtain an explicit construction of a polarization for any $\ell \in \fancyg_N^*$ in general position; see Corollary \ref{cor: explicit polarization}.
	
	Let $\{X_1, \ldots, X_m\}$ be a strong Malcev basis of $\fancyg_N$. Then, the dual space $\fancyg_N^*$ is spanned, as a vector space, by the corresponding dual basis $\{\ell_1, \ldots, \ell_m\}$. For any $\ell \in \fancyg_N^*$ the coadjoint orbit $\mathcal{O}_\ell = (\coadjoint G_N)\ell$ of $\ell$ is a closed submanifold of $\fancyg_N^*$, so that, in particular, it has a well defined dimension. For any $j = 1, \ldots,m$ we set
	\begin{equation*}
		\fancyg_N^*(j)=\mathbb{R}\text{-span}\{\ell_{j+1}, \ldots, \ell_m\}
	\end{equation*}
	so that $\fancyg_N^*/\fancyg_N^*(j) = \mathbb{R}$-span$\{\ell_1, \ldots, \ell_j\}$. We define the dimension of $\mathcal{O}_\ell$ in $\fancyg^*/\fancyg^*(j)$ as the dimension of $\mathcal{O}_\ell/\fancyg^*(j)\subseteq \fancyg_N^*/\fancyg_N^*(j)$.
	\begin{remark}
		Let $\{X_1^N, \ldots, X_{m_N}^N, X_1^{N-1}, \ldots, X_{m_{N-1}}^{N-1}, \ldots, X_1^1, \ldots, X_{m_1}^1\}$ such that for all $k \in \{1, \ldots, N\}$ the elements $X_1^k, \ldots, X_{m_k}^k$ are a basis of the $k$-th layer of $\fancyg_N$. Then, for any $k = 1, \ldots, N$ and $m = 1, \ldots, m_k$ we have
        \[\fancyg_N^*/\fancyg_{N}^*(k, m) = \R\text{-span}\{l_1^N, \ldots, l_m^k\} \;.\]
	\end{remark}
	\begin{definition}{(Orbit in general position)}
		Let $\{X_1, \ldots, X_m\}$ be a strong Malcev basis of $\fancyg_N$ and $\{\ell_1, \ldots, \ell_m\}$ the corresponding dual basis. Let $\ell \in \fancyg_N^*$ and $\mathcal{O}_\ell$ its coadjoint orbit. 
		\begin{itemize}[leftmargin=2em]
			\item[i)]  The orbit $\mathcal{O}_\ell$ has maximal dimension in $\fancyg_N^*/\fancyg_N^*(j)$ if for any other $\ell' \in \fancyg_N^*$ we have $\operatorname{dim}(\mathcal{O}_{\ell'}/\fancyg_N^*(j)) \leq \operatorname{dim}(\mathcal{O}_\ell/\fancyg_N^*(j))$.\\
			\item[ii)] 	The orbit $\mathcal{O}_\ell$ is in \textit{general position} or is a \textit{generic orbit}, if it is an orbit of maximal dimension in $\fancyg_N^*/\fancyg_N^*(j)$ for all $j \in \{1,\ldots, m\}$. Similarly, we will say that $\ell$ is in general position if its orbit $\mathcal{O}_\ell$ is.
		\end{itemize}
	\end{definition}
	\begin{example}\label{ex: general orbits of g2 pt1}
		As in Example \ref{ex: explicit orbits of g2} we consider $\fancyg_2(\R^d)$, with strong Malcev basis $\{X_{1,2}, X_{1,3}, \ldots, X_{d-1,d},X_1, \ldots, X_d\}$ . There, we have seen that the orbit of some $\ell = \sum_i \alpha_i \ell_i + \sum_{i,j} \beta_{i,j} \ell_{i,j} \in \fancyg_2^*$ is given by
		\begin{align*}
			\mathcal{O}_\ell &= \left\{\sum_{i=1}^d(\alpha_i+\sum_{j=1}^dg_{j}\beta_{i,j})\ell_i+\sum_{1\leq i<j\leq d}\beta_{i,j}\ell_{i,j} \;:\; (g_{1}, \ldots, g_d) \in \mathbb{R}^d\right\} \;.
		\end{align*}
		Clearly, the dimension of $\mathcal{O}_\ell$ is smaller or equal $d$. For any $1 \leq m<n\leq d$ we have that $\fancyg_2^*/\fancyg_{2}^*((m, n)) = \R$-span$\{l_{1,2}, \ldots,l_{m,n}\}$ so that
		\begin{align*}
			\dim\left(\mathcal{O}_\ell/\fancyg_{2}^*((m, n))\right)= \dim\Big\{\sum_{\substack{i\leq m,\\j\leq n}}\beta_{i,j}\ell_{i,j}\Big\} = 0 \;. \;
		\end{align*}
		Further, for any $m = 1, \ldots, d$, we have $\fancyg_2^*/\fancyg_{2}^*(m) = \R$-span$\{l_{1,2}, \ldots,l_{d-1,d}, l_1, \ldots, l_m\}$, so that
		\begin{align*}
			\dim\left(\mathcal{O}_\ell/\fancyg_{2}^*(m)\right)&= \dim\left\{\sum_{i=1}^m(\alpha_i+\sum_{j=1}^dg_{j}\beta_{i,j})\ell_i+\sum_{i,j}\beta_{i,j}\ell_{i,j} \;:\; (g_{1}, \ldots, g_d) \in \mathbb{R}^d\right\}\\
			&=\dim\left\{\sum_{i=1}^m\left(\sum_{j=1}^dg_{j}\beta_{i,j}\right)\ell_i \;:\; (g_{1}, \ldots, g_d) \in \mathbb{R}^d\right\}\\
			&=\operatorname{rank} B_\ell(m) \;,
		\end{align*}
		where $B_\ell(m) \in \R^{m\times d}$ is defined by $(B_\ell(m))_{i,j} = \beta_{i,j} = \ell([X_i,X_j])$. By definition, an orbit is in general position if the dimension of $\mathcal{O}_\ell / \mathfrak{g}_2^*(m)$ is maximal for all $m = 1, \ldots, d$. From the previous computation, this condition is equivalent to requiring that the rank of the matrix $B_\ell(m)$ is maximal for each $m$. In particular, if there exists an element $\ell$ such that $B_\ell(d)$ is invertible, then
		\[\dim\left(\mathcal{O}_\ell/\fancyg_{2}^*(m)\right) = m \quad
		\text{for all} \quad m = 1, \ldots, d\]
		and any other $\tilde{\ell} \in \mathfrak{g}_2^*$ is in general position if and only if $B_{\tilde{\ell}}(d)$ is invertible. However, at this point, it is not clear whether such an $\ell$ always exists or, more generally, what the maximal attainable rank of the matrix $B_\ell(d)$ is. The next proposition addresses this problem for arbitrary dimensions.
	\end{example}
	\begin{proposition}\label{prop: existence of invertible matrix}
		Let $\{X_1^N, \ldots, X_{m_N}^N, \ldots,X_1^1, \ldots, X_{m_1}^1\}$ be a strong Malcev basis for $\fancyg_N$ such that for all $k \in \{1, \ldots, N\}$ the elements $X_1^k, \ldots, X_{m_k}^k$ are a basis of the $k$-th layer of $\fancyg_N$. For $k \in \{1,\ldots, \lfloor N/2 \rfloor\}$ define the matrix $B_{\ell}^k \in \mathbb{R}^{m_k\times m_k}$ by
		\[(B_{\ell}^k)_{i,j} = \ell([X_i^k,X_j^{N-k}]) \quad i,j \in \{1, \ldots, m_k\} \;.\]
		Then, there is an $\ell \in \fancyg_N^*$ such that:
		\begin{itemize}
			\item[i)] If $N/2 \notin \mathbb{N}$, or $N/2 \in \mathbb{N}$ with $m_{N/2} \in 2\mathbb{N}$, then $B_{\ell}^k$ is invertible for all $k \leq N/2$.
			\item[ii)] If $N/2 \in  \mathbb{N}$ with $m_{N/2} \in 2\mathbb{N}+1$, then $B_{\ell}^k$ is invertible for all $k<N/2$ and $B_{\ell}^{N/2}$ has rank $m_{N/2}-1$, where the first $m_{N/2}-1$ rows of $B_{\ell}^{N/2}$ are linearly independent.
		\end{itemize}
	\end{proposition}
	\begin{remark}
		The reason for the differentiation into the two cases above is that if $N$ is even, then the matrix $B_{\ell}^{N/2} \in \mathbb{R}^{m_{N/2}\times m_{N/2}}$ is skew symmetric. Hence, if $ m_{N/2}$ is odd the matrix can have rank at most  $m_{N/2}-1$.
	\end{remark}
	\begin{proof}
		Before we begin with the proof we make the following observation.
		Let $k \neq N/2$, $X_i^k \neq 0$ be any basis element of the $k$-th layer, and $X_{j_1}^{N-k}, \ldots,X_{j_n}^{N-k}$ be distinct basis elements of the layer ${N-k}$. Then, the elements $[X_i^k,X_{j_1}^{N-k}], \ldots, [X_i^k,X_{j_n}^{N-k}]$ are linearly independent. Since, assuming there are $\alpha_1,\ldots, \alpha_n$ such that
		\begin{align*}
			0 = \sum_{s = 1}^{n}\alpha_s[X_i^k,X_{j_s}^{N-k}] = \sum_{s = 1}^{n}[X_i^k,\alpha_s X_{j_s}^{N-k}] = [X_i^k,\sum_{s = 1}^{n}\alpha_s X_{j_s}^{N-k}] \;,
		\end{align*}
		then $\sum_{s = 1}^{n}\alpha_s X_{j_s}^{N-k} = 0$ and therefore $\alpha_{1}= \ldots = \alpha_n = 0$.\\
		
		\noindent Note that the determinant of $B_{\ell}^k$ is a polynomial of the variables $\ell([X_i^k,X_j^{N-k}])$, or in other words, of the variables $\ell([X_i^N])$, since $[X_i^k,X_j^{N-k}] \in (R^d)^{\otimes N}$.
		We begin by proving that for all $k \neq N/2$, $\det B_{\ell}^k \not\equiv 0$, which implies that for every $k$ there is an $\ell$ such that $\det B_{\ell}^k \neq 0$.\\
		For $1< m \leq m_k$ let $B_{\ell}^k(m,m)$ be the $m\times m$ submatrix of $B_{\ell}^k$ containing only the first $m$ rows and columns. We show by induction that for all such $m$ $\det B_{\ell}^k(m,m) \not\equiv 0$.\\
		Let $m=2$. Then, $B_{\ell}^k(2,2)$ has the form
		\[B_{\ell}^k(2,2)=
		\begin{pmatrix}
			\ell([X_1^k, X_1^{N-k}]) & \ell([X_1^k, X_2^{N-k}]) \\
			\ell([X_2^k, X_1^{N-k}]) & \ell([X_2^k, X_2^{N-k}])
		\end{pmatrix}
		\]
		The elements $[X_1^k, X_1^{N-k}]$ and $[X_2^k, X_1^{N-k}]$ are linearly independent, so that we can set $\ell([X_1^k, X_1^{N-k}]) = \alpha_{1}\neq 0$ and $\ell([X_2^k, X_1^{N-k}])=0$. The element $[X_2^k, X_2^{N-k}]$ is either a linear combination of $[X_1^k, X_1^{N-k}]$ and $[X_2^k, X_1^{N-k}]$, so that  $\ell([X_2^k, X_2^{N-k}])=\alpha_2 = r\cdot\alpha_{1}\neq 0$ for some $r \in \mathbb{R}$; or it is linearly independent of $[X_1^k, X_1^{N-k}]$ and $[X_2^k, X_1^{N-k}]$, so that we can simply set $\ell([X_2^k, X_2^{N-k}])=\alpha_2 \neq 0 $. In both cases we obtain
		\[B_{\ell}^k(2,2)=
		\begin{pmatrix}
			\alpha_1 & *\\
			0& \alpha_2
		\end{pmatrix}
		\]
		which is invertible whenever $\alpha_{1}\alpha_2 \neq 0$. In particular, this means that $\det B_{\ell}^k(2,2) \not\equiv 0$.\\
		Let now $m>2$. Then, $B_{\ell}^k(m,m)$ has the form
		\NiceMatrixOptions{cell-space-limits = 2pt}
		\[B_{\ell}^k(m,m) =
		\begin{pNiceArray}{ccc|c}
			\Block{3-3}<\large>{B_{\ell}^k(m-1,m-1)} &&& \ell([X_1^k, X_m^{N-k}])\\
			&&& \vdots\\
			&&& \ell([X_{m-1}^k, X_m^{N-k}])\\
			\hline 
			\ell([X_m^k, X_{1}^{N-k}]) & \cdots & \ell([X_m^k, X_{m-1}^{N-k}]) &\ell([X_m^k, X_{m}^{N-k}])
		\end{pNiceArray}
		\]
		We set $\alpha_i = [X_i^k, X_m^{N-k}]$ for $i\leq m$. Then, due to linear independence of the elements $[X_i^k, X_m^{N-k}]$, the vector $\alpha = (\alpha_1,\ldots,\alpha_m)$ can take any value in $\mathbb{R}^m$. Further we set $\beta_j = [X_m^k, X_j^{N-k}]$ for $j<m$ where we note that the $\beta_j$ may depend on $\alpha$. Hence, $B_{\ell}^k(m,m)$ takes the form
		\NiceMatrixOptions{cell-space-limits = 2pt}
		\[B_{\ell}^k(m,m) =
		\begin{pNiceArray}{cW{c}{1cm}c|c}
			\Block{3-3}{B_{\ell}^k(m-1,m-1)} &&& \alpha_{1} \\
			&&& \vdots\\
			&&& \alpha_{m-1}\\
			\hline 
			\beta_1 & \cdots & \beta_{m-1} & \alpha_{m}
		\end{pNiceArray}
		\]
		The determinant of $B_{\ell}^k(m-1,m-1)$ is a polynomial in the variables $\alpha_i$ (and possibly other variables independent of $\alpha$) and by induction assumption it is not the $0$ polynomial. In particular there is some linear combination of the rows of $B_{\ell}^k(m-1,m-1)$ which equals the vector $\beta = (\beta_1,\ldots,\beta_{m-1})$. Note that this linear combination may depend on $\alpha$. Subtracting this linear combination from the last row we obtain that $B_{\ell}^k(m,m)$ has the same determinant as
		\[\begin{pNiceArray}{cW{c}{2cm}c|c}
			\Block{3-3}{B_{\ell}^k(m-1,m-1)} &&& \alpha_{1} \\
			&&& \vdots\\
			&&& \alpha_{m-1}\\
			\hline 
			0 & \cdots & 0 & A
		\end{pNiceArray}
		\]
		with $A = \alpha_{m}+\sum_{j=1}^{m-1}\lambda_j(\alpha,\beta)\alpha_j$ for some $\lambda_j(\alpha,\beta) \in \mathbb{R}$. Clearly $A \not\equiv 0$ since the ``monomial'' $\alpha_{m}$ doesn't appear in $\sum_{j=1}^{m-1}\lambda_j(\alpha,\beta)\alpha_j$. Hence, by induction, we obtain that $\det B_{\ell}^k(m,m) = \det\left( B_{\ell}^k(m-1,m-1) \right) \cdot \det A \not\equiv 0$.\\
		
		\noindent We now consider the case $k=N/2$. Similarly as in the previous case we show by induction that for all $1\leq m \leq (m_{N/2})/2$ the determinant of $B_{\ell}^{N/2}(2m,2m)$ is not the $0$ polynomial.\\
		For $m= 1$ we have that 
		\[B_{\ell}^{N/2}(2,2)=
		\begin{pmatrix}
			0 & \ell([X_1^{N/2}, X_2^{N/2}]) \\
			-\ell([X_1^{N/2}, X_2^{N/2}]) & 0
		\end{pmatrix}
		\]
		which is invertible whenever $\ell([X_1^{N/2}, X_2^{N/2}]) \neq 0$. \\
		Let $m>1$. Similarly as before, we set $\alpha_i = \ell([X_i^{N/2}, X_{2m}^{N/2}])$ and $\beta_i = \ell([X_i^{N/2}, X_{2m-1}^{N/2}])$ so that $\alpha = (\alpha_{1},\ldots,\alpha_{m-1}) \in \mathbb{R}^{m-1}$, and $\beta = (\beta_1,\ldots,\beta_{m-1}) \in \mathbb{R}^{m-1}$ may depend on $\alpha$. Then, $B_{\ell}^{N/2}(2m,2m)$ has the form
		\NiceMatrixOptions{cell-space-limits = 2pt}
		\[B_{\ell}^{N/2}(2m,2m)=
		\begin{pNiceArray}{cW{c}{1cm}c|c|c}
			\Block{3-3}{B_{\ell}^k(2(m-1),2(m-1))} &&& \beta_1 & \alpha_{1} \\
			&&& \vdots & \vdots\\
			&&& \beta_{2(m-1)} & \alpha_{2(m-1)}\\
			\hline 
			-\beta_1 & \cdots & -\beta_{2(m-1)} & 0 & \alpha_{2m-1}\\
			\hline
			-\alpha_{1} & \cdots &-\alpha_{2(m-1)}& -\alpha_{2m-1} &0
		\end{pNiceArray}
		\]
		The first $2m-1$ rows and columns of $B_{\ell}^{N/2}(2m,2m)$ form a skew symmetric matrix of odd dimension which is therefore not invertible. Further, by induction, $\det B_{\ell}^k(2(m-1),2(m-1)) \not\equiv 0$. Hence, applying, the Gauß algorithm to the last two rows, we obtain that $B_{\ell}^{N/2}(2m,2m)$ has the same determinant as
		\NiceMatrixOptions{cell-space-limits = 2pt}
		\[
		\begin{pNiceArray}{cW{c}{3cm}c|c|c}
			\Block{3-3}{B_{\ell}^k(2(m-1),2(m-1))} &&& \beta_1 & \alpha_{1} \\
			&&& \vdots & \vdots\\
			&&& \beta_{2(m-1)} & \alpha_{2(m-1)}\\
			\hline 
			0 & \cdots & 0 & 0 & A\\
			\hline
			0 & \cdots & 0 & B &0
		\end{pNiceArray}
		\]
		with $A = \alpha_{2m-1}+ \sum_{i=1}^{2(m-1)}\lambda_i(\alpha,\beta)\alpha_i$ and $B = -\alpha_{2m-1}+\sum_{i=1}^{2(m-1)}\psi_i(\alpha,\beta)\beta_i$ for some $\lambda_i(\alpha,\beta), \psi_i(\alpha,\beta) \in \mathbb{R}$. Clearly $A, B \not\equiv 0$ and the matrix has full rank whenever $A, B \neq 0$.\\
		
		\noindent We have now shown that $\det B_{\ell}^k$ (resp. $\det B_{\ell}^{N/2}(m_{N/2}-1,m_{N/2}-1)$) is not trivial, as a polynomial in the variables $\ell(X_i^N)$, for every $k=1, \ldots, N$. But then, the product of these polynomials is also nontrivial, so that there must be some $\ell$ such that 
		\begin{equation*}
			\det B_{\ell}^1 \cdot \ldots \cdot \det B_{\ell}^{\lfloor N/2\rfloor} \neq 0
		\end{equation*}
		(resp. $\det B_{\ell}^1 \cdot \ldots \cdot \det B_{\ell}^{N/2-1} \det B_{\ell}^{N/2}(m_{N/2}-1,m_{N/2}-1)\neq 0$ if $m_{N/2} \in 2\mathbb{N}+1$).
	\end{proof}
	The following result is a straightforward consequence of the previous proposition and will be useful in our characterization of the orbits in general position.
	\begin{corollary}\label{cor: max rank of blm}
		Let $\{X_1^N, \ldots, X_{m_N}^N, \ldots,X_1^1, \ldots, X_{m_1}^1\}$ be a strong Malcev basis for $\fancyg_N$ such that for all $k \in \{1, \ldots, N\}$ the elements $X_1^k, \ldots, X_{m_k}^k$ are a basis of $k$-th layer of $\fancyg_N$. For $k \in \{1,\ldots, N-1\}$ and $m \in \{1,\ldots,m_{N-k}\}$ define the matrix $B_{\ell}^k(m) \in \mathbb{R}^{m_k\times m}$ by
		\[(B_{\ell}^k(m))_{i,j} = \ell([X_i^k,X_j^{N-k}]) \quad i \in \{1, \ldots, m_k\}, \;j \in \{1, \ldots, m\} \;.\]
		Then, for any $m \in \{1, \ldots, m_{N-k}\}$
		\begin{align*}
			&\operatorname{dim}(k,m) \coloneqq \max_{\ell \in \fancyg_{N}^*}\left(\operatorname{rank}B_{\ell}^k(m)\right)=
			\min\left\{m_k,m_{N-k},m\right\} \text{, if } k \neq N/2\\
			&\operatorname{dim}(N/2,m) \coloneqq \max_{\ell \in \fancyg_{N}^*}\left(\operatorname{rank}B_{\ell}^{N/2}(m)\right) = \begin{cases}
				\min\left\{m_{N/2},m\right\} &\text{ if } m_{N/2} \text{ is even}\\
				\min\left\{m_{N/2}-1,m\right\} &\text{ if } m_{N/2} \text{ is odd}\\
			\end{cases}
		\end{align*}
		and there is an $\ell \in \fancyg_{N}^*$ such that $\operatorname{rank}B_{\ell}^k(m) = \dim(k,m)$ for all $k,m$.\\
		Further, the following hold for any $k < N/2$ :
		\begin{itemize}
			\item[i)] $\operatorname{rank}B_{\ell}^k(m) = \dim(k,m)$ for all $m = 1, \ldots, m_{N-k}$ if and only if $\rank B_{\ell}^k = \dim(k,m_k)$, i.e. if and only if $B_{\ell}^k$ is invertible.
			\item[ii)] If $B_{\ell}^k$ is invertible, then $\operatorname{rank}B_{\ell}^{N-k}(m) = \dim(N-k,m)$ for all $m = 1, \ldots, m_{k}$.
		\end{itemize}
	\end{corollary}
	\begin{proof} Fix a $k < N/2$. The implication ``$\Rightarrow$'' in $i)$ holds trivially because $m_k \leq m_{N-k}$. \\
		Now, assume $\rank B_{\ell}^k = \dim(k,m_k)$. We may assume without loss of generality that $m < m_k$, because if $m\geq m_k$, then $\operatorname{rank} B_{\ell}^k(m) \geq \operatorname{rank} B_{\ell}^k(m_k) = \rank B_{\ell}^k = \dim(k,m_k) = \dim(k,m)$.\\
		For all $m< m_k$ the matrix $B_{\ell}^k(m) \in \mathbb{R}^{m_k\times m}$ is the submatrix of $B_{\ell}^k$ containing only the first $m$ columns of $B_{\ell}^k$. Since, by assumption, $B_{\ell}^k$ is invertible, these columns are linearly independent so that 
		$\operatorname{rank} B_{\ell}^k(m) = m$ for all $m \leq m_k$. Hence, $\operatorname{rank} B_{\ell}^k(m) \geq \min\left\{m_k,m_{N-k},m\right\}$ for all $m \leq m_k$, which concludes the proof of $i)$.\\
		Now, consider the matrix $B_{\ell}^k(m,m_{N-k}) \in \R^{m \times m_{N-k}}$ containing only the first $m$ rows of $B_{\ell}^k(m_{N-k})$. Then, it is easy to check that $B_{\ell}^{N-k}(m) = -B_{\ell}^k(m,m_{N-k})^T$.\\
		Assuming $B_{\ell}^k$ is invertible, then by part $i)$ $\rank B_{\ell}^k(m_{N-k}) = \min\left\{m_k,m_{N-k}\right\} = m_k$, but since $B_{\ell}^k(m_{N-k}) \in \R^{m_k \times m_{N-k}}$ this means that the rows of $B_{\ell}^k(m_{N-k})$ must be linearly independent. Hence, we obtain that $\rank B_{\ell}^{N-k}(m) = \rank B_{\ell}^k(m,m_{N-k}) = m = \min\left\{m_k,m_{N-k}\right\}$, which shows $ii)$. \\
		Note that by definition, $B_{\ell}^k(m) \in \mathbb{R}^{m_k\times m}$ for any $k = 1, \ldots, N$ and $m = 1, \ldots, m_k$ , so that $\operatorname{rank} B_{\ell}^k(m)\leq \min\left\{m_k,m\right\} = \min\left\{m_k,m_{N-k},m\right\}$. With this observation, $i)$ and $ii)$ together prove in particular the first part of the corollary for all $k \neq N/2$, since, by Proposition \ref{prop: existence of invertible matrix}, we know that there is an $\ell \in \fancyg_N^*$ such that $B_{\ell}^k$ is invertible for all $k <N/2$. The case $k=N/2$ follows easily by using the same arguments as above adapted to $k =N/2$.        
	\end{proof}
	\begin{remark}
		Similar statements to $i)$ and $ii)$ in Corollary \ref{cor: max rank of blm} also hold for $k = N/2$, where, as in Proposition \ref{prop: existence of invertible matrix}, one would have to differentiate the cases $m_{N/2}$ even and odd. Since for the proof of Theorem \ref{thm: characterization of gen orbits} we only need $i)$ and $ii)$ for $k \neq N/2$, we omit them to avoid cluttering in the Corollary's statement.
	\end{remark}
	\begin{remark}
		Note that for all $k\neq N/2$ we have that 
		$\operatorname{dim}(N-k,m) = \min\left\{m_k,m_{N-k},m\right\} = \operatorname{dim}(k,m)$, in particular $\operatorname{dim}(k,m)$ is symmetric around $N/2$.
	\end{remark}
	\begin{example}\label{ex: general orbits of g2 pt2}
		Combining our observations from Example \ref{ex: general orbits of g2 pt1} and Proposition \ref{prop: existence of invertible matrix} we see that if $d$ is even, then $\ell \in \fancyg_2^*(\R^d)$ is in general position if and only if $B_{\ell}^1$ is invertible. On the other hand, if $d$ is odd, then $\ell$ is in general position if and only if the first $d-1$ rows of $B_{\ell}^1$ are linearly independent.
	\end{example}
	\begin{example}\label{ex: general orbits of g_3}
		Consider the group $G_3 = G_3(\R^d)$ for $d>3$. By example \ref{ex: explicit orbit of g3} we know that the coadjoint orbit of any $\ell = \sum_{i,j}\alpha_i^j\ell_i^j \in \fancyg_3^*$ is given by
		\begin{align*}
			\mathcal{O}_{\ell}= &\Bigg\{\sum_{i=1}^d \left(\alpha_i^1 + \sum_{j=1}^{d} g_j^1\ell([X_i^1,X_j^1]) + \sum_{j=1}^{m_2} g_j^2\ell([X_i^1,X_j^2]) + \frac{1}{2} \sum_{j,s = 1}^d g_j^1g_s^1\ell([X_j^1[X_i^1,X_s^1]]) \right)\ell_i^1\\
			&+ \sum_{i=1}^{m_2}\left(\alpha_i^2 + \sum_{j=1}^{d} g_j^1\ell([X_i^2,X_j^1])\right)\ell_i^2 + \sum_{i=1}^{m_3} \alpha_i^3\ell_i^3 \;\Big| \; (g_1^1, g_2^1, \ldots, g_{m_2}^2) \in \R^{d+m_2}\Bigg\} \;.
		\end{align*}
		Hence, for any $m \in \{1, \ldots,m_2\}$ we have that
		\begin{align*}
			\dim\left(\mathcal{O}_{\ell}/\fancyg_{2}^*(2, m)\right) &=   \dim\Bigg\{\sum_{i=1}^{m}\left(\alpha_i^2 + \sum_{j=1}^{d} g_j^1\ell([X_i^2,X_j^1])\right)\ell_i^2 + \sum_{i=1}^{m_3} \alpha_i^3\ell_i^3 \;\Big| \; (g_1^1,\ldots, g_{d}^1) \in \R^{d}\Bigg\} \\
			&=   \dim\Bigg\{\sum_{i=1}^{m}\left(\sum_{j=1}^{d} g_j^1\ell([X_i^2,X_j^1])\right)\ell_i^2 \;\Big| \; (g_1^1,\ldots, g_{d}^1) \in \R^{d}\Bigg\}\\ 
			&= \operatorname{rank} -B_{\ell}^{1}(m)^T = \operatorname{rank} B_{\ell}^{1}(m) \;.
		\end{align*}
		By Corollary \ref{cor: max rank of blm} the maximal attainable rank of $B_{\ell}^{1}(m)$ is $\min\{d,m_2,m\} = \min\{d,m\}$. Hence, a necessary condition for $\ell$ to be in general position is that $\operatorname{rank} B_{\ell}^{1}(m) = \min\{m,d\}$ for all $m \leq m_2$. By the same corollary this is equivalent to $B_{\ell}^{1} = B_{\ell}^{1}(d)$ being invertible. Let now $m \in \{1, \ldots, d\}$. Then,
		\begin{align*}
			\dim&\left(\mathcal{O}_{\ell}/\fancyg_{2}^*(1, m)\right) \\
            =& \dim \Bigg\{\sum_{i=1}^m \left(\sum_{j=1}^{d} g_j^1\ell([X_i^1,X_j^1]) + \sum_{j=1}^{m_2} g_j^2\ell([X_i^1,X_j^2]) + \frac{1}{2} \sum_{j,s = 1}^d g_j^1g_s^1\ell([X_j^1[X_i^1,X_s^1]]) \right)\ell_i^1\\
			&+ \sum_{i=1}^{m_2}\left(\sum_{j=1}^{d} g_j^1\ell([X_i^2,X_j^1])\right)\ell_i^2 \;\Big| \; (g_1^1, g_2^1, \ldots, g_{m_2}^2) \in \R^{d+m_2}\Bigg\} \;.
		\end{align*}
		By the previous observation, we know that the term
		$\sum_{i=1}^{m_2}\left(\sum_{j=1}^{d} g_j^1\ell([X_i^2,X_j^1])\right)\ell_i^2$
		spans a space of dimension at most $d$, so that clearly $\dim\left(\mathcal{O}_{\ell}/\fancyg_{2}^*(1, m)\right) \leq m +d$. Further, we observe that the dimension of the space spanned by the term
		\begin{equation}\label{eq: ex. space spanned g_3}
			\sum_{i=1}^m \left(\sum_{j=1}^{m_2} g_j^2\ell([X_i^1,X_j^2])\right)\ell_i^1
		\end{equation}
		is precisely the rank of the matrix $-B_{\ell}^2(m)^T$. By Corollary \ref{cor: max rank of blm} we know that if $B_{\ell}^1$ is invertible, then $\rank B_{\ell}^2(m)$ is maximal and equal to $m$ for all $m =1, \ldots, d$. 
		Note that the space spanned by (\ref{eq: ex. space spanned g_3}) only depends on $g_1^2, \ldots, g_m^2$ while the space spanned by the $\ell_i^2$ only depends on $g_1^1,\ldots,g_d^1$. Hence, if $B_{\ell}^1$ is invertible, we have that  $\dim\left(\mathcal{O}_{\ell}/\fancyg_{2}^*(1, m)\right) \geq m +d$. Since $B_{\ell}^1$ being invertible was also a necessary condition for $\ell$ to be in general position, we obtain in total that $\ell$ is in general position if and only if $B_{\ell}^1$ is invertible. In this case, for any $m = 1, \ldots, d$
		\begin{align*}
			\dim\left(\mathcal{O}_{\ell}/\fancyg_{2}^*(2, m)\right) = m \quad \text{and} \quad
			\dim\left(\mathcal{O}_{\ell}/\fancyg_{2}^*(1, m)\right) = d + m \;.
		\end{align*}
	\end{example}
	\begin{remark}
		The reason why we needed to exclude the case $G_3 = G_3(\R^2)$ in the previous example is that it represents, in some sense, an exceptional case. In our argument, we used implicitly that for any $k<N/2$ $\dim p_k(\fancyg_N)< \dim p_{N-k}(\fancyg_N)$. However, this assumption is not true for $\fancyg_3(\R^2)$, since $\dim \R^2 = 2 > 1 = \dim [\R^2]^2$. We deal with this example  explicitly in the appendix \ref{sec: General orbits of G3R2}.
	\end{remark}
	In the previous two examples, the problem of characterizing the orbits in general position could be reduced to verifying the invertibility of the matrix $B_{\ell}^1$. As the following theorem shows, these examples, represent only a special case of a more general result.
	\begin{theorem}\label{thm: characterization of gen orbits}
		Let $\{X_1^N, \ldots, X_{m_N}^N, \ldots,X_1^1, \ldots, X_{m_1}^1\}$ be a strong Malcev basis for $\fancyg_N \neq \fancyg_3(\R^2)$ such that for all $k \in \{1, \ldots, N\}$ the elements $X_1^k, \ldots, X_{m_k}^k$ are a basis of $k$-th layer of $\fancyg_N$, and let $\ell\in \fancyg_{N}^*$. Then, with the notation of Corollary \ref{cor: max rank of blm}, the following are equivalent.
		\begin{itemize}
			\item[i)] The dimension of $\mathcal{O}_{\ell}$ is maximal in $\fancyg_{N}^*/\fancyg_{N}^*((N-k, m)) = \mathbb{R}$-span$\{\ell_1^N, \ldots, \ell_{m}^{N-k}\}$ for all $k \in \{1, \ldots,N-1\}$ and $m \in \{1, \ldots, m_{N-k}\}$, i.e. $\ell$ is in general position.
			\item[ii)] The rank of the matrix $B_{\ell}^{k}(m)$ is maximal among elements of $\fancyg_{N}^*$ for all $k \in \{1, \ldots,N-1\}$ and $m \in \{1, \ldots, m_{N-k}\}$.
			\item[iii)] The rank of the matrix $B_{\ell}^{k}$ is maximal,  among elements of $\fancyg_{N}^*$ for all $k \in \{1, \ldots,\lfloor N/2\rfloor\}$. I.e. $B_{\ell}^{k} = B_{\ell}^{k}(m_k)$ satisfies property $i)$ resp. $ii)$ from Proposition \ref{prop: existence of invertible matrix}.
		\end{itemize}
		If $\ell$ satisfies any of the conditions above, then for all $k \in \{1, \ldots,N-1\}$
		\begin{align}\label{eq: dimension formual for generic orbits}
			\operatorname{dim}\left(\mathcal{O}_{\ell} / \fancyg_N^*(N - k, m)\right) =
			\sum_{s=1}^{k-1}\operatorname{dim}(s,m_s)+ \operatorname{dim}(k,m)
		\end{align}
	\end{theorem}
	\begin{proof}
		We assume $N > 3$ since we already computed the cases $N=2,3$ explicitly in examples \ref{ex: general orbits of g2 pt2} and \ref{ex: general orbits of g_3} respectively.
		Let $\ell = \sum_{s=1}^{N}\sum_{j=1}^{m_s}\alpha_j^s\ell_j^s \in \fancyg_{N}^*$.\\
		The equivalence $ii) \iff iii)$ is precisely what we showed in Corollary \ref{cor: max rank of blm}.\\   
		We now prove the equivalence $i) \iff ii)$ together with the dimension formula (\ref{eq: dimension formual for generic orbits}) by induction over $k$.\\
		$1)\;$ $k = 1$. Let $X = \sum_{s=1}^{N}\sum_{j=1}^{m_s}g_j^sX_j^s \in \fancyg_{N}$. Then, for all $i \in \{1, \ldots, m_{N-1}\}$
		\begin{align*}
			(\Adjoint \operatorname{exp}-X)X_i^{N-1} &= X_i^{N-1} + \sum_{n=1}^{\infty}\frac{1}{n!}(\operatorname{ad}-X)^n X_i^{N-1}\\
			&=  X_i^{N-1} + [X_i^{N-1}, X] = X_i^{N-1} + \sum_{j=1}^{m_1}g_j^1[X_i^{N-1},X_j^1]
		\end{align*}
		Hence, for any $m \in \{1, \ldots, m_{N-1}\}$ the orbit of $\ell$ in $\mathbb{R}$-span$\{\ell_1^N, \ldots, \ell_{m}^{N-1}\}$ is given by
		\begin{align*}
			\mathcal{O}_{\ell} / \fancyg_N^*(N-1, m) = \Bigg\{ 
			& \sum_{i=1}^{m} \left(\alpha_i^{N-1} + \sum_{j=1}^{m_1} \ell([X_i^{N-1}, X_j^1])g_j^1\right) \ell_i^{N-1} \\
			& + \sum_{i=1}^{m_N} \alpha_i^N \ell_i^N \; \Big| \; (g_1^1, \ldots, g_{m_1}^1) \in \mathbb{R}^{m_1} \Bigg\}
		\end{align*}
		From this equation we see that the dimension of $\mathcal{O}_{\ell} / \fancyg_N^*(N-1, m)$ is precisely the dimension of the image, or in other words the rank, of the $m \times m_1$ matrix with entries $\ell([X_i^{N-1}, X_j^1])$. Note that this matrix is precisely $-(B_{\ell}^{1}(m))^T$. Hence, we obtain that the dimension of $\mathcal{O}_{\ell} / \fancyg_N^*(N-1, m)$ is maximal for all $m$ if and only if the rank of $B_{\ell}^{1}(m)$ is. By Corollary \ref{cor: max rank of blm}, we know that the maximal rank of $B_{\ell}^{1}(m)$ is given by $\dim(1,m)$ and it is attained if and only if $B_{\ell}^1$ is invertible\vspace{5pt}.\\  
		\noindent $2)$ Assume now $k>1$. Let $X = \sum_{s=1}^{N}\sum_{j=1}^{m_s}g_j^sX_j^s \in \fancyg_{N}$, then
		\begin{alignat*}{3}
			(\Adjoint \operatorname{exp} - X)X_i^{N-k} 
			& = X_i^{N-k} 
			& & + \sum_{n=1}^{\infty} \frac{1}{n!} (\operatorname{ad} -X)^n X_i^{N-k} \\
			& = X_i^{N-k} 
			& &+ [X_i^{N-k}, X] 
			+ \sum_{n=2}^{\infty} \frac{1}{n!} (\operatorname{ad} -X)^n X_i^{N-k} \\
			& = X_i^{N-k} 
			&& + \sum_{s=1}^{k} \sum_{j=1}^{m_s} g_j^s [X_i^{N-k}, X_j^s] \\
			&& & + \sum_{n=0}^{\infty} \frac{1}{(n+2)!} (\operatorname{ad} -X)^n [[X_i^{N-k}, X], X] \;.
		\end{alignat*}
		Note that the very last sum only depends on the $g_j^s$ with $s \leq k-1$, since for all $n \in \mathbb{N}_0$ $(\operatorname{ad}-X_j^{s_1})^n [[X_i^{N-k}, X_j^{s_2}], X_j^{s_3}] = 0$ if any of the $s_i \geq k$. Hence, if we set
		\begin{equation*}
			R(g_1^1,g_2^1, \ldots, g_{m_{k-1}}^{k-1}) = \sum_{s=1}^{k-1} \sum_{j=1}^{m_s} g_j^s \; \ell([X_i^{N-k}, X_j^s]) + \sum_{n=0}^{\infty} \frac{1}{(n+2)!}\; \ell\left((\operatorname{ad} -X)^n [[X_i^{N-k}, X], X]\right)
		\end{equation*}
		the orbit of $\ell$ in $\fancyg_N^*(N-k, m)$ has the form
		\begin{align}\label{eq: korbit}
			\mathcal{O}_{\ell} / \fancyg_N^*(N-k, m) = &\Bigg\{ 
			\sum_{i=1}^{m} \left(\alpha_i^{N-k} + \sum_{j=1}^{m_k} g_j^k \ell([X_i^{N-k}, X_j^k]) + R(g_1^1,g_2^1, \ldots, g_{m_{k-1}}^{k-1})\right) \ell_i^{N-k} \notag \\
			& + \sum_{s = 1}^{k-1} \sum_{i=1}^{m_s} r_i^{N-s} \ell_i^{N-s} \; \Big| \; (g_1^1,g_2^1, \ldots, g_{m}^k) \in \mathbb{R}^{m_1+m_2+\ldots + m} \Bigg\}
		\end{align}
		where the $r_i^{N-s}$ are some real coefficients that depend on the $g_i^s$.\\
		We need to differentiate two cases. First, assume $1<k\leq N/2$. Then,
		$\operatorname{dim}(N-s,m) = \min\left\{m_s,m\right\}$ for all $s< k$.\\
		Now, the term
		$\sum_{s = 1}^{k-1} \sum_{i=1}^{m_s} r_i^{N-s} \ell_i^{N-s}$ only depends on $(g_1^1,g_2^1 \ldots, g_{m_{k-1}}^{k-1}) \in \mathbb{R}^{m_1+m_2+\ldots + m_{k-1}}$.
		Since, as we mentioned before, also $R(g_1^1,g_2^1, \ldots, g_{m_{k-1}}^{k-1})$ only depends on $(g_1^1,g_2^1 \ldots, g_{m_{k-1}}^{k-1})$ we can rewrite the orbit as
		
		\[\mathcal{O}_{\ell} / \fancyg_N^*(N - k, m) = W_{\ell}^{k-1} + V_{\ell}^k \quad \text{where}\]
		\begin{alignat*}{3}
			W_{\ell}^{k-1} &=  \Bigg\{
			&& \sum_{i=1}^{m} \left(\alpha_i^{N-k}+ R(g_1^1,g_2^1, \ldots, g_{m_{k-1}}^{k-1})\right) \ell_i^{N-k}\\
			&&& + \sum_{s = 1}^{k-1} \sum_{i=1}^{m_s} r_i^{N-s} \ell_i^{N-s} \; \Big| \; (g_1^1,g_2^1 \ldots, g_{m_{k-1}}^{k-1}) \in \mathbb{R}^{m_1+m_2+\ldots + m_{k-1}} \Bigg\} \quad \text{and} \\
			V_{\ell}^k &= \Bigg\{ 
			&&\sum_{i=1}^{m} \left(\sum_{j=1}^{m_k} g_j^k \ell([X_i^{N-k}, X_j^k])\right) \ell_i^{N-k} \; \Big| \; (g_1^k, \ldots, g_{m_k}^k) \in \mathbb{R}^{m_k} \Bigg\} \;.
		\end{alignat*}
		
		Clearly $\max_{\ell' \in \fancyg_{N}^*}\left(\operatorname{dim}W_{\ell'}^{k-1}\right)\leq m_1+\ldots + m_{k-1}$. On the other hand, by induction assumption, if $\operatorname{rank}B_{\ell}^s(m_s)$ is maximal for all $s<k$, then $$\operatorname{dim}W_{\ell}^{k-1} = \operatorname{dim}(N-1,m_1)+\ldots + \operatorname{dim}(N-k+1,m_{k-1})=m_1+\ldots + m_{k-1} \;.$$
		Further, as before, $\dim(V_{\ell}^k) =\operatorname{rank}(B_{\ell}^k(m))$. Hence, if the rank of $B_{\ell}^k(m)$ is also maximal we obtain that
		\begin{align*}
			\operatorname{dim}\left(\mathcal{O}_{\ell} / \fancyg_N^*(N - k, m)\right)
			&= \sum_{s=1}^{k-1}\operatorname{dim}(N-s,m_s)+ \operatorname{dim}(N-k,m)\\
			&= \sum_{s=1}^{k-1}m_s+ \max_{\ell' \in \fancyg_{N}^*}\left(\rank B_{\ell'}^k(m)\right)\\
			&=\max_{\ell' \in \fancyg_{N}^*}\left(\dim W_{\ell'}^{k-1} + \dim V_{\ell}^k\right) = \max_{\ell' \in \fancyg_{N}^*}\big(\operatorname{dim}\left(\mathcal{O}_{\ell'} / \fancyg_N^*(N - k, m)\right) \big)
		\end{align*}
		Conversely, if the dimension of $\mathcal{O}_{\ell} / \fancyg_N^*(N - k, m)$ is maximal, then so is $\dim(V_{\ell}^k) = \operatorname{rank}(B_{\ell}^k(m))$\vspace{5pt}.\\
		We now assume $k>N/2$. \\
		Then, $N-k<N/2$ and $m_k>m_{N-k}$. If $\ell$ is in general position, then we know by induction that $B_{\ell}^{N-k}$ must have maximal rank. By Corollary \ref{cor: max rank of blm} $ii)$ this implies that $B_{\ell}^{k}(m)$ has maximal rank for all $m$, i.e. $\rank B_{\ell}^{k}(m) = \dim(k,m)$.\\
		Conversely, assume that $\operatorname{rank}B_{\ell}^{N-s}(m)=\dim(N-s,m)$ for all $s$. By induction the term $ \sum_{s = 1}^{k-1} \left(\sum_{i=1}^{m_s} r_i^{N-s} \ell_i^{N-s} \right)$ in equation (\ref{eq: korbit}) spans a subspace of dimension $\sum_{s =1}^{k-1}\dim(N-s,m_s)$. Further, for all $m \in \{1, \ldots, m_{N-k}\}$ the term
		\[\sum_{i=1}^{m} \left(\alpha_i^{N-k} + \sum_{j=1}^{m_k} g_j^k \ell([X_i^{N-k}, X_j^k]) + R(g_1^1,g_2^1, \ldots, g_{m_{k-1}}^{k-1})\right) \ell_i^{N-k}\]
		spans a subspace of dimension less or equal $m =\min\left\{m,m_{N-k}\right\} =\min\left\{m,m_{N-k},m_k\right\} = \dim(N-k,m)$ and again equality holds if $B_{\ell}^{k}(m)$ has rank $m$.
		This concludes the proof of $i) \iff ii)$ and of the dimension formula.
	\end{proof}
	\begin{remark}\label{rem: general dimensions jump}
		Dimension formula (\ref{eq: dimension formual for generic orbits}) shows an important behavior of the maximal dimension of orbits which will play a crucial role in the Fourier inversion formula. For any $k<N/2$ the dimension in $\fancyg_{N}^*/\fancyg_{N}^*(N-k,m)$ of orbits in general position increases by one for every basis element we add to the space up to the $m_k$-th element, and remains constant after that. In other words for all $k<N/2$ we have that
		\begin{align}
			&\max_{\ell \in \fancyg_{N}^*}\operatorname{dim}\left(\mathcal{O}_{\ell}/\fancyg_{N-k}^*(N-k,m+1)\right) = \max_{\ell \in \fancyg_{N}^*}\operatorname{dim}\left(\mathcal{O}_{\ell}/\fancyg_{N-k}^*(N-k,m)\right) +1 \quad \text{if } m<m_k \label{eq: beahviour of dimensions 1}\\
			&\max_{\ell \in \fancyg_{N}^*}\operatorname{dim}\left(\mathcal{O}_{\ell}/\fancyg_{N-k}^*(N-k,m_k+j)\right) = \max_{\ell \in \fancyg_{N}^*}\operatorname{dim}\left(\mathcal{O}_{\ell}/\fancyg_{N-k}^*(N-k,m_k)\right) \quad \text{for all } j>0 \;.\label{eq: beahviour of dimensions 2}
		\end{align}
		The same formula remains true if $k = N/2$ and $m_{N/2} \in 2\mathbb{N}$, whereas in the case $m_{N/2} \in 2\mathbb{N}+1$ one needs to substitute $m_{N/2}$ with $m_{N/2}-1$ in (\ref{eq: beahviour of dimensions 1}) and (\ref{eq: beahviour of dimensions 2}). If $k>N/2$ then $m_{N-k}<m_k$ so that the formula becomes
		\[\max_{\ell \in \fancyg_{N}^*}\operatorname{dim}\left(\mathcal{O}_{\ell}/\fancyg_{N}^*(N-k,m+1)\right) = \max_{\ell \in \fancyg_{N}^*}\operatorname{dim}\left(\mathcal{O}_{\ell}/\fancyg_{N}^*(N-k,m)\right) +1 \quad \text{if } m<m_{N-k}\]
	\end{remark}
	Combining Theorem \ref{thm: vergne} with the characterization of general orbits provided by Theorem \ref{thm: characterization of gen orbits} allows us to compute a polarization for any $\ell \in \fancyg_N^*$ in general position.
	\begin{corollary}\label{cor: explicit polarization}
		Let $\ell \in \fancyg_N^*$ be a functional such that $\mathcal{O}_{\ell}$ is in general position.
		\begin{itemize}
			\item[i)] If $N \in \mathbb{N}$ is odd, then 
			$$\fancym_{\ell} = \bigoplus_{k=\lceil N/2 \rceil}^N [\mathbb{R}^d]^k = \fancyg_N/\fancyg_{\lfloor N/2\rfloor}$$
			is a polarization for $\ell$.
			\item[ii)]  If $N \in \mathbb{N}$ is even. For any $m \in \{1,\ldots,m_{N/2}\}$ let $B_{\ell}^{N/2}(m,m)$ be the submatrix of $B_{\ell}^{N/2}$ containing only the first $m$ rows and columns. Then, 
			$$\fancym_{\ell} = \sum_{m=1}^{m_{N/2}}\ker B_{\ell}^{N/2}(m,m) \; \oplus \bigoplus_{k=\lceil N/2 \rceil}^N [\mathbb{R}^d]^k = \sum_{m=1}^{m_N}\ker B_{\ell}^{N/2}(m,m) \oplus \fancyg_N/\fancyg_{N/2}$$
			is a polarization for $\ell$.
		\end{itemize}
	\end{corollary}
	\begin{proof}
		Let $\{X_1^N, \ldots, X_{m_N}^N, \ldots,X_1^1, \ldots, X_{m_1}^1\}$ be a our usual strong Malcev basis for $\fancyg_N$ and let $\mathfrak{i}^{(k,n)} = \mathbb{R}$-span$\{X_1^N, \ldots, X_n^k\}$. Then, by definition, the sequence 
		\[(0) \subseteq \mathfrak{i}^{(N,1)} \subseteq \ldots \subseteq \mathfrak{i}^{(N,m_N)} \subseteq \mathfrak{i}^{(N-1,1)} \subseteq \ldots \subseteq \mathfrak{i}^{(1,m_1)}\]
		is a sequence of ideals as in Theorem \ref{thm: vergne}.
		We prove the statement by showing that the subalgebra $\fancym_{\ell}$ above coincides with the polarization constructed in Theorem \ref{thm: vergne}, which we denote by $\tilde{\fancym}_{\ell}$.\\
		First, note that for all $k_1,k_2>N/2$,  $[\mathfrak{i}^{(k_1,n_1)},\mathfrak{i}^{(k_2,n_2)}] = 0$ for any $n_1,n_2$. In particular
		$$r(\ell_{(k,n)})\coloneqq \left\{Y \in \mathfrak{i}^{(k,n)} \; : \; \ell([X,Y]) = 0 \text{ for all } X \in \mathfrak{i}^{(k,n)}\right\}= \mathfrak{i}^{(k,n)}$$ 
		for all $k>N/2$, so that 
		$$\fancym_{\ell} = \bigoplus_{k=\lceil N/2 \rceil}^N [\mathbb{R}^d]^k \subseteq \tilde{\fancym}_{\ell} \;.$$
		We now show that $\tilde{\fancym}_{\ell} \subseteq \fancym_{\ell}$. \\
		Let $X = X^1+ \ldots +X^{N}$ with $X^i \in [\mathbb{R}^d]^i$ for all $i$ and such that $X^i \neq 0$ for at least one index $i <N/2$. Let $k$ be the first index such that $X^k \neq 0$, so that in particular $X = X^k+ \ldots +X^{N}$.\\ Assume there is an index $(s,n)$ such that $X \in r(\ell_{(s,n)})$. 
		Then, since $s$ must be smaller than $N/2$, $[\mathbb{R}^d]^{N-k} \subseteq \mathfrak{i}^{(s,n)}$ so that, in particular, $X$ must satisfy $\ell([X_i^{N-k},X]) = 0$ for all $i = 1, \ldots, m_{N-k}$. Since $X^k \in \rdbracket^k$ there are some $\alpha_1,\ldots, \alpha_{m_k} \in \mathbb{R}$ such that $X^k = \sum_{j} \alpha_jX_j^k$. Hence, we obtain that for all $i = 1, \ldots, m_{N-k}$
		\[0 = \ell([X_i^{N-k},X]) = \ell([X_i^{N-k},X^k]) =  \sum_{j=1}^{m_k}\alpha_j \ell([X_i^{N-k},X_j^k])\;.\]
		In particular, we have that 
		\[0 = B_{\ell}^{N-k}(m_k) \cdot
		\begin{pmatrix}
			\alpha_1 \\
			\alpha_2 \\
			\vdots \\
			\alpha_{m_k}
		\end{pmatrix}
		\]
		However, by Theorem \ref{thm: characterization of gen orbits} we know that, since $\ell$ is in general position $B_{\ell}^{N-k}(m_k)$ has rank $m_k$. But then we must have $\alpha_i = 0$ for all $i$, so that $X^k=0$, which is a contradiction to the choice of $X^k$. Hence,
		\[\fancym_{\ell} \subseteq \bigoplus_{k \geq N/2} [\mathbb{R}^d]^k \;.\]
		If $N$ is odd this already concludes the proof. Set $V_m = \mathbb{R}$-span$\{X_1^{N/2}, \ldots, X_m^{N/2}\}$. The case $N/2 \in \mathbb{N}$ follows by noticing that for any index $m \in \{1, \ldots, m_{N/2}\}$
		\begin{align*}
			r(\ell_{(N/2,m)})&= \left\{Y \in \mathfrak{i}^{(N/2,m)} \; : \; \ell([X,Y]) = 0 \text{ for all } X \in \mathfrak{i}^{(N/2,m)}\right\}\\
			&= \left\{Y \in V_m \; : \; \ell([X,Y]) = 0 \text{ for all } X \in V_m \right\} \oplus  \bigoplus_{k > N/2} [\mathbb{R}^d]^k\\
			&= \left\{Y \in V_m \; : \; X^T B_{\ell}^{N/2}(m,m)Y = 0 \text{ for all } X \in V_m \right\} \oplus  \bigoplus_{k > N/2} [\mathbb{R}^d]^k\\
			&= \left\{Y \in V_m \; : \;B_{\ell}^{N/2}(m,m)Y =0 \right\} \oplus  \bigoplus_{k > N/2} [\mathbb{R}^d]^k\\
			&= \ker B_{\ell}^{N/2}(m,m) \oplus  \bigoplus_{k > N/2} [\mathbb{R}^d]^k
		\end{align*}
	\end{proof}
	
	%
	\begin{remark}
		Note that if $N$ is odd, the polarization $\fancym_{\ell}$ provided by Corollary \ref{cor: explicit polarization} is independent of the functional $\ell$. Hence, the representation $\pi_{\ell}$ can be induced from the exact same subgroup for all functionals in general position. On the other hand, if $N$ is even, the polarization of two functionals in general position only differs by a subset of $\rdbracket^{N/2}$.
	\end{remark}
	Note that the only properties of the algebra $\fancyg_N$ we used in the proofs of this section are that $\dim p_k(\fancyg_N) \leq \dim p_{N-k}(\fancyg_N)$ for any $k \leq N/2$ and that $[v_k,w_{N-k}] \neq 0$ for any $v_k \in p_k(\fancyg_N)$ , $w_{N-k} \in p_{N-k}(\fancyg_N)$ with $v_k \neq \pm w_{N-k}$. In particular, the exact same proofs also hold for any graded Lie group\footnote{See definition \ref{def: graded}.} satisfying these properties. We therefore obtain the following slightly more general version of the results presented in this section.
	\begin{corollary}\label{cor: generalization of results 1}
		Let \( G \) be an \( N \)-step graded Lie group with Lie algebra \( \fancyg \), and suppose that for all \( i \leq j \) with \( i + j = N \), the following conditions hold:
		\begin{itemize}
			\item[i)] \( \dim V_i \leq \dim V_j \),
			\item[ii)] \( [v_i, w_j] \neq 0 \) for all \( v_i \in V_i \setminus \{0\} \), \( w_j \in V_j \setminus \{0\} \) with \( v_i \neq \pm w_j \).
		\end{itemize}
		Then, Proposition \ref{prop: existence of invertible matrix}, Corollary \ref{cor: max rank of blm}, Theorem \ref{thm: characterization of gen orbits}, Remark \ref{rem: general dimensions jump} and Corollary \ref{cor: explicit polarization}  remain valid if $G_N$ and $\fancyg_N$ are replaced by $G$ and $\fancyg$ respectively. In the case of Corollary \ref{cor: explicit polarization} one must additionally replace $\rdbracket^k$ with $p_k(\fancyg)$, i.e. the $k$-th layer of $\fancyg$.
	\end{corollary}
	\begin{example}
		The group $T_1^N\left(\R^d\right)$ generated by the tensor algebra $T_0^N\left(\R^d\right)$ defined in section $\ref{sec: signatures}$ satisfies the properties in the previous corollary.
	\end{example}
	\newpage
	\subsection{The Fourier inversion formula}
	
	Let $\pi_\ell: G_N \to \mathcal{F}_{\ell}$ be a unitary representation of $G_N$. Then, for any $f \in L^1(G_N)$ we define the continuous linear operator $\pi_\ell(f)$ on $\mathcal{F}_{\ell}$ by
	\[\pi_\ell(f) = \int_{G_N} f(x) \pi_\ell(x) \;dx \;.\]
	Note that this operator is in general only weakly defined. That is, for any $u \in \mathcal{F}_{\ell}$ the value of $\pi_\ell(f)u$ is specified by its inner product with some arbitrary $v \in \mathcal{F}_{\ell}$, specifically
	\[\langle \pi_\ell(f)u,v\rangle = \int_{G_N} f(x)\langle \pi_\ell(x)u,v\rangle \;dx\;.\]
	One can easily see that $\langle \pi_\ell(f)u,v\rangle \leq ||f||_1||u||||v||$ for any $u,v \in \mathcal{F}_{\ell}$ so that $\pi_\ell(f)$ is in fact bounded with operator norm $||\pi_\ell(f)|| \leq ||f||_1$. The idea for generalizing the concept of a Fourier transform to the group of signatures is to consider operators of the form $\pi_\ell(f)$ for unitary irreducible representations $\pi_\ell$ of $G_N$.\\
	Recall from Remark \ref{rem: characterization of dual} that the unitary dual of $G_N$ can be written as
	\[\hat{G}_N = \left\{[\ell] \;|\; \ell \in \fancyg_N^*\right\} \;,\]
	where $[\ell]$ is the equivalence class determined by the relation $\ell \cong \ell' :\iff \mathcal{O}_{\ell} = \mathcal{O}_{\ell'}$. Using this characterization, we define the Fourier transform as a function on ${G}_N$ as follows.
	\begin{definition}\label{def: fourier trafo}
		Let $f \in L^1(G_N)$. Then, we define the Fourier transform of $f$ as the operator valued mapping
		\[\hat{f}: [\ell] \mapsto \hat{f}(\ell) = \pi_{\ell}(f) = \int_G f(x) \pi_{\ell}(x) \;dx \quad \text{for all } [\ell] \in \hat{G} \;.\]
	\end{definition}
	\begin{remark}
		Recall that the representation $\pi_{\ell}$ was defined in the previous section as the representation induced by the functional $\ell$ \textit{up to equivalence}. Since we know that if $[\ell] = [\ell']$ then $\pi_{\ell} \cong \pi_{\ell'}$, the map in the previous definition is in fact well defined.
	\end{remark}
		
		Let $\{X_{1}, \ldots, X_{m}\}$ be a strong Malcev basis of $\fancyg_N$. For $j \in \{1, \ldots, m\}$ we let $d_j$ be the maximal dimension of orbits in $\fancyg^*/\fancyg^*(j)$. Where $d_m$ is the maximal dimension of orbits in $\fancyg_N^*$.
		Let $S = \{j_1, \ldots, j_{2k}\}$ with $j_1< \ldots< j_{2k}$ be the set of indices $j$ such that $d_j <d_{j-1}$ i.e. $S$ is the set of indices $j$ at which the maximal dimension of orbits in $\fancyg_N^*/\fancyg_N^*(j)$ increases (by one). We define $(\fancyg^*)_S = \mathbb{R}$-$\operatorname{span}\{\ell_i \;:\; i \in S\}$ and $(\fancyg^*)_T = \mathbb{R}$-$\operatorname{span}\{\ell_i \;:\; i \in T\}$ with $T =  \{1, \ldots, m\}\backslash S$. Note that $(\fancyg^*)_S \times (\fancyg^*)_T = \fancyg_N^*$.
		\begin{remark}\label{rem: choice of representant}
			The reason for the interest in the set $S$ defined above is that it allows a canonical choice of representative for the equivalence classes of functionals in general position in $\hat{G}_N = \left\{[\ell] \;|\; \ell \in \fancyg_N^*\right\}$. Specifically, if $\ell \in \fancyg_N^*$ is in general position, its orbit $\mathcal{O}_\ell$ intersects the subspace $\{0\} \times (\fancyg_N^*)_T$ in exactly one point \cite[Theorem~3.1.9]{corwin1990representations}. Hence, if $U \subseteq \hat{G}_N$ denotes the set of functionals in general position, we have that
			\begin{equation*}
				U \cap  \hat{G}_N \cong U \cap (\fancyg_N^*)_T \;.
			\end{equation*}
		\end{remark}
	
		\begin{theorem}{(General Fourier inversion formula)}\label{thm: fourier inversion}\\
			Let, $\{X_1, \ldots, X_m\}$ be an orthonormal strong Malcev basis of $\fancyg_N$ and $\{\ell_1, \ldots, \ell_{m}\}$ the corresponding dual basis. Let $U = \left\{(\alpha_{1}, \ldots, \alpha_{m}) \in \mathbb{R}^{m} \; | \; \ell = \sum_{i=1}^{m}\alpha_i \ell_i \text{ is in general position} \right\}$ be the set of functionals in general position and $S= \{j_1, \ldots, j_{2k}\}$ be as above. Further, define $D(\ell) \in \mathbb{R}^{2k \times 2k}$ as
			$$(D(\ell))_{i,s} = \ell([X_{j_i},X_{j_s}]) \quad \text{for } i,s \in \{1, \ldots, 2k\} \;.$$
			Then, for any Schwartz function $f \in \mathcal{S}(G)$
			\[f(x) = \int_{U\cap (\fancyg_N^*)_T}\sqrt{\operatorname{det}D(\ell)}\operatorname{Tr}\left( \pi_{\ell}(x)^{-1}\pi_{\ell}(f)\right)dl \;,\]
			where $dl$ is the Lebesgue measure on $(\fancyg^*)_T = \mathbb{R}$-$\operatorname{span}\{\ell_i \;:\; i \in T\}$.
		\end{theorem}
		\begin{proof}
			See \cite[Theorem~4.3.9]{corwin1990representations}.
		\end{proof}
		\begin{corollary}{(Plancherel Theorem)}\label{cor: plancherel formula}
			With the notation of Theorem \ref{thm: fourier inversion} we have
			\[||f(x)||_2^2 = \int_{U\cap (\fancyg_N^*)_T}\sqrt{\operatorname{det}D(\ell)}\; ||\pi_{\ell}(f)||_{\text{HS}}^2 \;dl \;,\]
			where $|| \cdot||_2$ is the $L^2$ norm and $||\cdot||_{\text{HS}}$ is the Hilbert-Schmidt norm.
		\end{corollary}
		
		By Theorem \ref{thm: fourier inversion} we see that the key elements to compute the Fourier inversion formula are the set $U$ of orbits in general position and the set $S$ of jump indices. We already characterized the orbits in general position in the previous section and, with the theory developed there, it is now easy to also describe the sets $S$ and $T$ explicitly for general $N$ and $d$.
		\begin{lemma}\label{lem: S and T}
			Let $\{X_{1}^N, \ldots, X_{m_{N}}^N, \ldots, X_1^1, \ldots, X_{m_1}^1\}$ be a strong Malcev basis of $\fancyg_N({\mathbb{R}^d}) \neq \fancyg_3(\R^2)$. Then, with the notation of Theorem \ref{thm: fourier inversion}
			\begin{align*}
				S &=\bigcup_{k=1}^{N-1}\left\{(k,1), \ldots, (k,\dim(k,m_k))\right\},
			\end{align*}
			where the index $(k,i)$ corresponds to the basis element $l_i^k$, and $\dim(k,m_k)$ is defined as in Corollary \ref{cor: max rank of blm}.
			In particular, 
			\begin{equation*}
				T = 
                \begin{cases}
                \;\displaystyle\bigcup_{k > N/2}^N \left\{(k, m_{N-k}+1), \ldots, (k, m_k)\right\}, 
                & \text{if } N/2 \notin \mathbb{N} \text{ or } m_{N/2} \in 2\mathbb{N}, \\[1.5ex]
                \;\displaystyle\bigcup_{k > N/2}^N \left\{(k, m_{N-k}+1), \ldots, (k, m_k)\right\} \cup \left\{(N/2, m_{N/2})\right\}, 
                & \text{if } m_{N/2} \in 2\mathbb{N} + 1.
                \end{cases}
			\end{equation*}
			Further, the set of orbits in general position of $g_N$ is given by
			\[U = \left\{\ell \in \fancyg_{N}^* \;|\, \operatorname{rank} B_{\ell}^k = \dim(k,m_k) \text{ for all } k \in \{1, \ldots,\lfloor N/2\rfloor\}\right\}\;.\]
		\end{lemma}
		\begin{proof}
			This is precisely what we showed in Theorem \ref{thm: characterization of gen orbits} and Remark \ref{rem: general dimensions jump}.
		\end{proof}
		\begin{remark}\label{rem: dimension of S}
		From the previous lemma, we see that for any $k<N/2$ the set $S$ contains the indices of the first $m_k = \dim\left([\mathbb{R}^d]^k\right)$ basis elements of $[\mathbb{R}^d]^k$ and $[\mathbb{R}^d]^{N-k}$. Further if $N/2 \in \mathbb{N}$, $S$ also contains either all, or all but one indices of the basis elements of  $[\mathbb{R}^d]^{N/2}$. Hence, we obtain that
        \begin{equation*}
				(\fancyg_N^*)_T =
                    \begin{cases}
                    \displaystyle\bigoplus_{k > N/2}^N \R\text{-span}\{X_{m_{N-k}+1}^k, \ldots, X_{m_k}^k\}, 
                    & \text{if } N/2 \notin \mathbb{N} \text{ or } m_{N/2} \in 2\mathbb{N}, \\[1.5ex]
                    \displaystyle\bigoplus_{k > N/2}^N \R\text{-span}\{X_{m_{N-k}+1}^k, \ldots, X_{m_k}^k\} 
                    \oplus \R\text{-span}\{X_{m_{N/2}}^{N/2}\}, 
                    & \text{if } m_{N/2} \in 2\mathbb{N}+1.
                    \end{cases}
			\end{equation*}
		In particular, the number of indices in $S$ is given by
		\begin{equation*}\label{eq: number of element in S}
			|S| =
                \begin{cases}
                \displaystyle 2 \cdot \dim(\fancyg_{\lfloor N/2 \rfloor}), 
                & \text{if } N/2 \notin \mathbb{N} \text{ or } m_{N/2} \in 2\mathbb{N}, \\[1.2ex]
                \displaystyle 2 \cdot \dim(\fancyg_{N/2}) - 1, 
                & \text{if } m_{N/2} \in 2\mathbb{N} + 1.
                \end{cases}
		\end{equation*}
        Using Remark \ref{rem: choice of representant}, the Fourier transform can now be simply interpreted as a function defined on the subspace $(\fancyg_N^*)_T$.
        \end{remark}
		\begin{remark}
			Note that the set $(\fancyg_N^*)\backslash U$ is a $0$-measure set with respect to the Lebesgue measure. Hence, by setting $n \coloneqq |T| = \dim(\fancyg_{N})-|S|$ the inversion theorem can also be written as
			\[f(x) = \int_{\mathbb{R}^n}\sqrt{\operatorname{det}D(\ell(\alpha))}\operatorname{Tr}\left( \pi_{\ell(\alpha)}(x)^{-1}\pi_{\ell(\alpha)}(f)\right)d(\alpha_{j}^k)_{(k,j)\in T} \;,\]
			where $\ell(\alpha)= \sum_{(k,j)\in T}\alpha_j^kl_j^k$. In particular
			\[||f||_2^2 = \int_{\mathbb{R}^n}\sqrt{\operatorname{det}D(\ell(\alpha))}||\;\pi_{\ell(\alpha)}(f)||_{HS}^2\; d(\alpha_{j}^k)_{(k,j)\in T} \;.\]
		\end{remark}
		\begin{remark}\label{rem: generalization of results 2}
			Note that since the statements in this section follow directly from the results established in the previous section, they also extend to the class of graded Lie groups described in Corollary \ref{cor: generalization of results 1}.
		\end{remark}
		
		\begin{remark}\label{rem: kernel of fourier transform}
			It can be shown $-$ see \cite[Proposition~4.2.2 in]{corwin1990representations} $-$ that for any $f \in \mathcal{S}(G_N)$ the operator $\pi_{\ell}(f)$ is a trace class integral operator with kernel
			\[K_{f}: G_N/{H_{\ell}} \times G_N/{H_{\ell}} \to \C, \quad K_f(x,y) = \int_{H_{\ell}} f(x u y^{-1})e^{il(\log(u))} du\]
			where $H_{\ell} = \exp(\fancyh_{\ell})$ and $\fancyh_{\ell}$ is a polarization of $\ell$. Accordingly, the action of $\pi_{\ell}(f)$ on a function $\phi \in \mathcal{F}_{\ell}$ is given, up to right-multiplication with elements of $H_\ell$, by
			\[(\pi_{\ell}(f)\phi)(x) = \int_{G_N/H_{\ell}}K_f(x,y)\phi(y) dy\]
			This expression is useful for two main reasons. First, since the operator is completely determined by its kernel, it allows us to identify the Fourier transform of $f$ with a complex-valued function, instead of a linear operator on some function space. Second, it provides a method for computing the trace of $\pi_{\ell}(x)^{-1}\pi_{\ell}(f)$, since the trace of such an integral operator is given by 
			\[\operatorname{Tr} \pi_{\ell}(x)^{-1}\pi_{\ell}(f) = \operatorname{Tr}\pi_{\ell}(L_{x^{-1}}f) =  \int_{G_N/H_{\ell}}K_{(L_{x^{-1}}f)}(y,y) dy =  \int_{G_N/H_{\ell}}K_f(xy,y) dy\;.\]
			Note that Corollary \ref{cor: explicit polarization} provides an explicit construction of a polarization $\fancyh_{\ell}$, so that all integrals above are in fact computable.
		\end{remark}
		
		The Plancherel formula presented in Corollary \ref{cor: plancherel formula} implicitly shows that for any Schwartz function $f \in \mathcal{S}(G_N)$ the operator $\pi_{\ell}(f)$ is Hilbert-Schmidt. Hence, by Theorem \ref{thm: fourier inversion} the map
		\[
		\mathcal{S}(G_N)\to \int_{\hat{G}}^{\oplus}HS(\mathcal{H}_{\ell})\; d\mu(\ell), \quad
		f \mapsto \int_{\hat{G}}^{\oplus}\pi_{\ell}(f)\; d\mu(\ell)
		\]
		is an isometry, where $\mu = \sqrt{B_{\ell}}dl$ and $HS(\mathcal{H}_{\ell})$ is the set of Hilbert-Schmidt operators on $\mathcal{H}_{\ell}$. Therefore, it extends to an isometry $L^2(G_N) \to \int_{\hat{G}}^{\oplus}HS(\mathcal{H}_{\ell})$. As the following result from \cite{corwin1990representations} shows, this map is surjective.
		\begin{theorem}
			Let $G$ be a nilpotent Lie group with Haar measure $dx$ and $\mu = \sqrt{D(\ell)}dl$ as in Theorem \ref{thm: fourier inversion}. Then, the map
			\[
			L^2(G_N)\to \int_{\hat{G}}^{\oplus}HS(\mathcal{H}_{\ell})\; d\mu(\ell)
			\]
			is a surjective isometry. 
		\end{theorem}
		\begin{example}\label{ex:explicit formula for g2}
			We consider the group $G_2(\mathbb{R}^d)$ for $d$ even. Let $\{X_{i}\}_{i=1,\ldots,d} \cup \{X_{i,j}\}_{1\leq i<j\leq d}$ be a basis of $\fancyg_2(\mathbb{R}^d)$. By examples \ref{ex: general orbits of g2 pt1} and \ref{ex: general orbits of g2 pt2} we know that the functionals in general position are precisely those $\ell = \sum \alpha_i \ell_i + \sum_{i,j} \beta_{i,j} \ell_{i,j}$ such that $B_{\ell}^1 = (B_{\ell}^1)_{i,j} = \beta_{i,j}$ is invertible. Further, we know that for such functionals we have that
			\[\dim\left(\mathcal{O}_{\ell}/\fancyg_{2}^*(m)\right) = m \quad
			\text{for all } m = 1, \ldots, d\]
			In particular, we obtain that $d_{i,j} = 0$ for all $1\leq i<j \leq d$ and $d_{j} = j$ for all $j = 1, \ldots, d$. Hence, $S=\{i\}_{i=1,\ldots,d}$, $T = \{(i,j)\}_{1\leq i<j\leq d}$ and $U \cap (\fancyg_{N}^*)_T =\left\{ (\beta_{i,j})_{1\leq i<j\leq d} \in \mathbb{R}^n \; | \; B \text{ is invertible} \right\}$, where $n = \frac{d(d-1)}{2}$. Further, the matrix $D(\ell)$ is given by:
			\[
			D(\ell) = \begin{pmatrix} 0 & \beta_{1,2} & \beta_{1,3} & \cdots & 	\beta_{1,d} \\ -\beta_{1,2} & 0 & \beta_{2,3} & \cdots & \beta_{2,d} \\ -\beta_{1,3} & -\beta_{2,3} & 0 &\ddots & \vdots \\ \vdots & \vdots & \ddots & \ddots & \beta_{d-1,d} \\ -\beta_{1,d} & -\beta_{2,d} & \cdots & -\beta_{d-1,d} & 0 \end{pmatrix}
			\]
			and the Fourier inversion formula becomes
			\[
			f(x) = \int_{\mathbb{R}^n}\sqrt{\operatorname{det}D(\beta_{1,2},\beta_{1,3},\ldots, \beta_{d-1,d})}\operatorname{Tr}\left( \pi_{\ell}(x)^{-1}\pi_{\ell}(f)\right)d\beta_{1,2}d\beta_{1,3}\ldots d\beta_{d-1,d} \;.
			\]
			In particular we have that
			\[
			\|f\|_2^2 = \int_{\mathbb{R}^n}\|\pi_{\ell}(f)\|_{\text{HS}}^2\sqrt{\operatorname{det}D(\beta_{1,2},\beta_{1,3},\ldots ,\beta_{d-1,d})}\;d\beta_{1,2}d\beta_{1,3}\ldots d\beta_{d-1,d}
			\]		
		\end{example}
		\begin{example}
			Let $3\leq d \in \mathbb{N}$. Then, by Lemma \ref{lem: dimension formula} $m_2 = \dim[\mathbb{R}^d]^2 = \frac{d^2-d}{2}$ and $m_3 = \dim[\mathbb{R}^d]^3 = \frac{d^3-d}{3}$. Let $\{X_1^3, \ldots,X_{m_3}^3, \ldots, X_1^1, \ldots,X_d^1\}$ be a strong Malcev basis of $\fancyg_{3}$. Then, by Lemma \ref{lem: S and T} and Remark \ref{rem: dimension of S}
			\begin{align*}
				&S = \{X_1^2, \ldots, X_d^2,X_1^1, \ldots, X_d^1\}\\
				&|S| = 2d\\
				&T = \{X_1^3, \ldots, X_{m_3}^3,X_{d+1}^2, \ldots, X_{m_2}^2\}\\
				&|T| = d+\frac{1}{2}(d^2-d)+\frac{1}{3}(d^3-d)-2d =\frac{1}{3}d^3+ \frac{1}{2}d^2-\frac{11}{6}d
			\end{align*}
			The Fourier inversion formula is then given by
			\[f(x) = \int_{\mathbb{R}^{|T|}}\sqrt{\operatorname{det}D(\ell(\alpha))}\operatorname{Tr}\left( \pi_{\ell(\alpha)}(x)^{-1}\pi_{\ell(\alpha)}(f)\right)d\alpha_1^3\ldots d\alpha_{m_3}^3 d\alpha_{d+1}^2\ldots d\alpha_{m_2}^2 \;.\]
		\end{example}
		
		\begin{example}
			We take a closer look at the case $d = N = 3$. Let $\{X_1, X_2, X_3\}$ be an orthonormal basis of $\mathbb{R}$. For brevity we write $X_{(i,j)} = [X_i,X_j]$ and $X_{(i,j,s)} = [X_i,[X_j,X_s]]$. Then,
			\begin{align*}
				&\{X_{(1,2)}, X_{(1,3)}, X_{(2,3)}\} \quad \text{and}\\
				&\{X_{(1,1,2)}, X_{(1,1,3)}, X_{(1,2,3)}, X_{(2,1,2)}, X_{(2,1,3)}, X_{(2,2,3)}, X_{(3,1,3)}, X_{(3,2,3)}\}.
			\end{align*}
			are a basis of $ [\mathbb{R}^3]^2$ and $ [\mathbb{R}^3]^3$ respectively.
			By the previous example we know that
			\begin{align*}
				&S = \{1,2,3, (1,2), (1,3),(2,3)\} \; \text{and}\\
				& T = \{(1,1,2), (1,1,3),(1,2,3),(2,1,2),(2,1,3),(2,2,3),(3,1,3),(3,2,3)\} 
			\end{align*}	
			We set $\alpha_{i,j,s} = \ell(X_{(i,j,s)})$, then the matrix $D(\ell(\alpha))$ is given by 
			\[D(\ell(\alpha))=
			\begin{pmatrix}
				\alpha_{1,1,2} & \alpha_{1,1,3} & \alpha_{1,2,3} \\
				\alpha_{2,1,2} & \alpha_{2,1,3} & \alpha_{2,2,3} \\
				\alpha_{2,1,3}-\alpha_{1,2,3} & \alpha_{3,1,3} & \alpha_{3,2,3}
			\end{pmatrix}\]
			so that the Fourier inversion formula becomes
			\[f(x) = \int_{\mathbb{R}^{8}}\sqrt{\operatorname{det}D(\ell(\alpha))}\operatorname{Tr}\left( \pi_{\ell(\alpha)}(x)^{-1}\pi_{\ell(\alpha)}(f)\right)d\alpha_{1,1,2} d\alpha_{1,1,3}\ldots d\alpha_{3,2,3}\;,\]
			where $\ell(\alpha) = \sum_{(i,j,s) \in T}\alpha_{i,j,s}\ell_{(i,j,s)}$.
		\end{example}

		\newpage
		\begin{appendices}
			\section{}\label{secA1}
			\subsection{Additional Results and Definitions}
			\begin{definition}\label{def: graded}
				We say that a Lie algebra $\mathfrak{g}$ is graded when it is endowed with a vector space decomposition
				$$
				\mathfrak{g}=\bigoplus_{j=1}^{\infty} V_j \quad \text { such that } \quad\left[V_i, V_j\right] \subset V_{i+j}
				$$
				and all but finitely many of the $V_j$ 's are $\{0\}$. We call the subspace $V_j$ the $j$-th layer of $\fancyg$. A Lie group is called graded when it is a connected, simply connected Lie group whose Lie algebra is graded. Further, we say that the group is $n$-step graded if $V_n \neq \{0\}$ and $V_k = 0$ for all $k>n$. Note that any ($n$-step) graded Lie group is in particular ($n$-step) nilpotent.
			\end{definition}
			\begin{lemma}{(Dimension formula)}\label{lem: dimension formula}
				Let $d, k \in \mathbb{N}$. Then, the dimension of $[\mathbb{R}^d]^k$ is given by 
				\[\operatorname{dim}[\mathbb{R}^d]^k = \frac{1}{k}\sum_{n|k}\mu(n)d^{\frac{k}{n}} \;,\]
				where the sum is taken over all divisors $n$ of $k$, and $\mu$ is the Möbius function defined as
				\[\mu(n)= \begin{cases}1 & \text { if } n=1 \\ (-1)^k & \text { if } n \text { is the product of } k \text { distinct primes } \\ 0 & \text { if } n \text { is divisible by a square }>1\end{cases} \;.\]
			\end{lemma}
			\begin{proof}
				See \cite{witt1937treue}.
			\end{proof}
			\begin{theorem}\label{thm: vergne}
				Let $(0) \subseteq \fancyg_1 \subseteq \fancyg_2 \subseteq \ldots \subseteq \fancyg_n$ be a chain of ideals in $\fancyg$ such that $\dim \fancyg_j = j$. For $\ell \in \fancyg^*$ define $r({l_j}) = \left\{Y \in \fancyg_j \; : \; \ell([X,Y]) = 0 \text{ for all } X \in \fancyg_j\right\}$. Then, $$\fancym_l = \sum_{j = 0}^n r(\ell_j)$$
				is a maximal subordinate subalgebra for $\ell$.
				\begin{proof}
					See \cite{vergne1971} or \cite[Theorem~1.3.5]{corwin1990representations}.
				\end{proof}
			\end{theorem}
			\subsection{General orbits of \texorpdfstring{$G_3(\R^2)$}{G3(R2)}}\label{sec: General orbits of G3R2}
			We wish to characterize the orbits in general position of $G_3 = G_3(\R^2)$. The reason why we need to consider this group separately is that it is, in some sense, a degenerate case. It is in fact the only situation when for $k<N/2$ $\dim p_k(G_N(\R^d))> \dim p_{N-k}(G_N(\R^d))$, since $\dim \R^2 = 2 > 1 = \dim [\R^2]^2$.
			
			Let $X_1, X_2$ be a basis of $\R^2$ and set $X_{ij} = [X_i,X_j]$ and $X_{ijk} = [X_i,[X_j, X_k]]$. Then, a strong Malcev basis of $G_3$ is given by
			$\{X_{1,1,2}, X_{2,1,2}, X_{1,2}, X_1, X_2\}$. By example \ref{ex: explicit orbit of g3} we know that for any $\ell = \alpha_1l_1+ \alpha_2l_2+ \alpha_{1,2}\ell_{1,2}+ \alpha_{112}\ell_{112} + \alpha_{212}\ell_{212} \in \fancyg_3^*$ the orbit of $\ell$ is given by
			\begin{equation}\label{eq: explicit orbit of G3R2}
				\begin{aligned}
					\mathcal{O}_{\ell} = \Bigg\{ 
					&\sum_{i=1}^2 \Bigg( \alpha_i
					+ \sum_{j=1}^{2} g_j\, \alpha_{i,j} 
					+g_{1,2}\, \alpha_{i,1,2} + \frac{1}{2} \sum_{j=1}^{2} \sum_{s=1}^{2} g_j g_s\, \alpha_{j,i,s}
					\Bigg) \ell_i \\
					&+ \left( \alpha_{1,2}- \sum_{j=1}^{2} g_j\, \alpha_{j,1,2}) \right) \ell_{1,2}
					+ \alpha_{112}\ell_{112} + \alpha_{212}\ell_{212} \;\Big|\; (g_1,g_2, g_{1,2}) \in \mathbb{R}^3\Bigg\}\\
					= \Bigg\{ 
					&\Big( \alpha_1 
					+ g_2\, \alpha_{1,2} 
					+ g_{1,2}\, \alpha_{1,1,2} 
					+ \frac{1}{2} \big( 
					g_1 g_2\, \alpha_{1,1,2}
					+ g_2 g_2\, \alpha_{2,1,2} 
					\big) \Big) \ell_1 \\
					+ &\Big( \alpha_2 
					+ g_1\, \alpha_{2,1} 
					+ g_{1,2}\, \alpha_{2,1,2} 
					+ \frac{1}{2} \big( 
					g_1 g_1\, \alpha_{1,2,1}
					+ g_2 g_1\, \alpha_{2,2,1} 
					\big) \Big) \ell_2 \\
					+ &\left( \alpha_{1,2} 
					- g_1\, \alpha_{1,1,2} - g_2\, \alpha_{2,1,2} \right) \ell_{1,2} 
					+ \alpha_{1,1,2}\ell_{1,1,2} + \alpha_{2,1,2}\ell_{2,1,2}
					\;\Big|\, (g_1,g_2, g_{1,2}) \in \mathbb{R}^3
					\Bigg\} \;.
				\end{aligned}
			\end{equation}
			
			From this expression one can easily see that if $\alpha_{1,1,2} \neq 0$ or $\alpha_{2,1,2} \neq 0$, then $\dim \mathcal{O}_{\ell} \geq 2$. Further, since coadjoint orbits are always even dimensional, we also have $\dim \mathcal{O}_{\ell} \leq 2$, so that equality must hold in this case. This guarantees that if $\alpha_{1,1,2} \neq 0$ or $\alpha_{2,1,2} \neq 0$ then $\mathcal{O}_{\ell}$ has maximal dimension. However, to check if the orbit is in general position we need to consider the quotient spaces of $\fancyg_3^*$. By (\ref{eq: explicit orbit of G3R2}) we have that
			\begin{align*}
				\dim \mathcal{O}_{\ell}/\R\text{-span}\{\ell_1,l_2\} = \dim \left\{\left( \alpha_{1,2} 
				- g_1\, \alpha_{1,1,2} - g_2\, \alpha_{2,1,2} \right) \ell_{1,2}
				\;\Big|\, (g_1,g_2) \in \mathbb{R}^2
				\right\} \leq 1
			\end{align*}
			and clearly equality holds whenever $\alpha_{1,1,2} \neq 0$ or $\alpha_{2,1,2} \neq 0$. Further
			\begin{align*}
				\dim \mathcal{O}_{\ell}/\R\text{-span}\{\ell_2\} = \dim \Big\{ 
				&\Big( \alpha_1 
				+ g_2\, \alpha_{1,2} 
				+ g_{1,2}\, \alpha_{1,1,2} 
				+ \frac{1}{2} \left( 
				g_1 g_2\, \alpha_{1,1,2}
				+ g_2 g_2\, \alpha_{2,1,2} 
				\right) \Big) \ell_1 \\
				+ &\left( \alpha_{1,2} 
				- g_1\, \alpha_{1,1,2} - g_2\, \alpha_{2,1,2} \right) l_{1,2}
				\;\Big|\, (g_1,g_2, g_{1,2}) \in \mathbb{R}^3
				\Big\} \;.
			\end{align*}
			From this expression we see that if $\alpha_{1,1,2} \neq 0$, then $\dim \mathcal{O}_{\ell}/\R\text{-span}\{\ell_2\} = 2$. However, if $\alpha_{1,1,2} = 0$, then
			\begin{align*}
				\dim \mathcal{O}_{\ell}/\R\text{-span}\{\ell_2\} = \dim \Big\{ 
				\Big( \alpha_1 
				+ g_2\, \alpha_{1,2}
				+ \frac{1}{2}
				g_2 g_2\, \alpha_{2,1,2} 
				\Big) \ell_1 
				+ \left( \alpha_{1,2} 
				- g_2\, \alpha_{2,1,2} \right) l_{1,2}
				\;\Big|\, g_2 \in \mathbb{R}\Big\} \leq 1 \;.
			\end{align*}
			Hence, we obtain that $\ell$ is in general position if and only if $\alpha_{1,1,2} \neq 0$.
			
			
			
			
		\end{appendices}
		
		\newpage
		\bibliography{sn-article}
	\end{document}